\documentclass{article}
\usepackage{amsmath}
\usepackage{amsthm}
\usepackage{graphics}
\usepackage{epsfig}

\newtheorem{theorem}{Theorem}[section]
\newtheorem{cor}[theorem]{Corollary}
\newtheorem{lem}[theorem]{Lemma}
\newtheorem{prp}[theorem]{Proposition}

\theoremstyle{definition}

\newtheorem{rem}[theorem]{Remark}

\newtheorem{example}[theorem]{Example}

\numberwithin{equation}{section}

\usepackage{amssymb}
\usepackage{amsfonts}
\usepackage{psfrag}
\usepackage{url}
\usepackage{mathtools}

\usepackage{stmaryrd}

\usepackage{amsfonts}
\usepackage{amsmath}
\usepackage{amssymb}

\usepackage{bbm}
\newcommand{\B}{\beta}
\newcommand{\I}{\mathbb{I}}
\newcommand{\N}{\mathbb{N}}
\newcommand{\R}{\mathbb{R}}

\newcommand{\cP}{\mathcal{P}}

\newcommand{\eqb}{\begin{equation}}
\newcommand{\eqe}{\end{equation}}

\newcommand{\lt}{\left}
\newcommand{\rt}{\right}

\newcommand{\rb}{\begin{rem}}
\newcommand{\re}{\end{rem}}

\usepackage{epsfig}
\usepackage{psfrag}
\usepackage{color}
\usepackage{graphicx}

\numberwithin{equation}{section}
\numberwithin{theorem}{section}

\begin{document}

\title{Best finite approximations of Benford's Law}

\author{Arno Berger and Chuang Xu\\[2mm] Mathematical and Statistical
  Sciences\\University of Alberta\\Edmonton, Alberta, {\sc Canada}}

\maketitle

\begin{abstract}
For arbitrary Borel probability measures with compact
support on the real line, characterizations are established of the best finitely
supported approximations, relative to three familiar probability
metrics (L\'{e}vy, Kantorovich, and Kolmogorov), given any number of
atoms, and allowing for additional constraints regarding weights or
positions of atoms. As an application, best (constrained or
unconstrained) approximations are identified for Benford's Law
(logarithmic distribution of significands) and other familiar
distributions. The results complement and
extend known facts in the literature; they also provide new rigorous
benchmarks against which to evaluate empirical observations regarding
Benford's Law.
\end{abstract}

\noindent
{\bf Keywords.} Benford's Law, best uniform approximation,
asymptotically best approximation, \\
\hspace*{19.0mm} L\'{e}vy distance, Kantorovich distance, Kolmogorov distance.

\noindent
{\bf MSC2010.} 60B10, 60E15, 62E15.

\medskip

\section{Introduction}
\label{intro}

Given real numbers $b>1$ and $x\ne 0$, denote by $S_b(x)$
the unique number in $[1,b[$ such that $|x|= S_b(x) b^k$ for
some (necessarily unique) integer $k$; for convenience, let $S_b(0) =
0$. The number $S_b(x)$ often is referred to as the {\em base}-$b$ {\em
  significand\/} of $x$, a terminology particularly well-established in
the case of $b$ being an integer. (Unlike in much of the literature
\cite{Ben,BH,Hill,Raimi}, the case
of integer $b$ does not carry special significance in this article.)
A Borel probability measure $\mu$ on $\R$ is {\em Benford base\/} $b$,
or $b$-{\em Benford\/} for short, if
\begin{equation}\label{ee1}
\mu \bigl(  \{ x\in \R : S_b(x) \le s\}\bigr) = \frac{\log s}{\log b}  \quad \forall
s \in [1,b[ \, ;
\end{equation}
here and throughout, $\log$ denotes the natural logarithm.
Benford probabilities (or random variables) exhibit many
interesting properties and have been studied extensively \cite{All,Fel,Hi2,Mil,Pin}. They provide
one major pathway into the study of {\em Benford's Law}, an
intriguing, multi-faceted phenomenon that attracts interest from a wide range of disciplines; see, e.g.,
\cite{BH} for an introduction, and \cite{Mil} for a panorama of recent developments.
Specifically, denoting by $\B_b$ the Borel probability measure with
$$
\B_b([1,s]) = \frac{\log s}{\log b} \quad \forall s\in [1,b[ \, ,
$$
note that $\mu$ is $b$-Benford if and only if $\mu \circ S_b^{-1} =
\B_b$.

Historically, the case of {\em decimal\/} (i.e., base-$10$)
significands has been the most prominent, with early empirical studies on
the distribution of decimal significands (or significant digits) going back to
Newcomb \cite{N} and Benford \cite{Ben}. If $\mu$ is 10-Benford, note
that in particular
\begin{equation}\label{ee2}
\mu \bigl(\{x\in \R : \mbox{\rm leading decimal digit of } x
= D \} \bigr) =  \frac{\log (1+D^{-1})}{\log 10 }\quad \! \forall D
=1 , \ldots, 9 \, .
\end{equation}
For theoretical as well as practical reasons, mathematical objects such as random
variables or sequences, but also concrete,
finite numerical data sets that conform, at least approximately, to
(\ref{ee1}) or (\ref{ee2}) have attracted much interest \cite{Dia,Knu,Raimi,Sch}. Time and again,
Benford's Law has emerged as a perplexingly prevalent phenomenon. One
popular approach to understand this prevalence seeks to establish
(mild) conditions on a probability measure that make
(\ref{ee1}) or (\ref{ee2}) hold with good accuracy, perhaps even
exactly \cite{BT,DL,Fel,GD,Pin}. It
is the goal of the present article to provide precise quantitative
information for this approach.

Concretely, notice that while a
finitely supported probability measure, such as, e.g., the empirical measure
associated with a finite data set \cite{BHM}, may conform to the {\em
  first-digit law\/} (\ref{ee2}), it cannot possibly satisfy
(\ref{ee1}) exactly. For such measures, therefore, it is natural to
quantify, as accurately as possible, the failure of equality in (\ref{ee1}), that is, the
discrepancy between $\mu \circ S_b^{-1}$ and $\B_b$. Utilizing three
different familiar metrics $d_*$ on probabilities (L\'evy,
Kantorovich, and Kolmogorov metrics; see Section \ref{sec2} for
details), the article does this in a systematic way: For every
$n\in \N$, the value of $\min_{\nu} d_*(\B_b, \nu)$ is identified,
where $\nu$ is assumed to be supported on no more than $n$
atoms (and may be subject to further restrictions such as,
e.g., having only atoms of equal weight, as in the case of empirical
measures); the minimizers of $d_*(\B_b, \nu)$ are also characterized explicitly.

The scope of the results presented herein, however, extends far beyond
Benford probabilities. In fact, a general theory of best (constrained or unconstrained)
$d_*$-approximations is developed. As far as the authors can tell, no such theories
exist for the L\'evy and Kolmogorov metrics --- unlike in the case of
the Kantorovich metric where it (mostly) suffices to rephrase pertinent known
facts \cite{GL,XB}. Once the general results are established, the desired
quantitative insights for Benford probabilities are but straightforward
corollaries. (Even in the context of Kantorovich distance, the study of
$\B_b$ yields a rare new, explicit example of an {\em optimal
  quantizer\/} \cite{GL}.) In
particular, it turns out that, under all the various constraints
considered here, the limit $Q_{*} = \lim_{n\to \infty} n \min_{\nu}
d_*(\B_b, \nu)$ always exists, is finite and positive, and can be
computed more or less explicitly. This greatly extends earlier results,
notably of \cite{BHM}, and also suggests that
$n^{-1}Q_{*}$ may be an appropriate quantity against which to evaluate
the many heuristic claims of closeness to Benford's
Law for empirical data sets found in
the literature \cite{BHMI,Mil,MT}.

The main results in this article, then, are existence proofs and
characterizations for the minimizers of $d_{*}(\mu, \nu)$ for arbitrary
(compactly supported) probability measures $\mu$, as provided by
Theorems \ref{thm3e}, \ref{prop3f}, \ref{prop4a}, \ref{thloc-K}, and
\ref{thwei-K} (where additional constraints are imposed on the sizes or locations of the atoms
of $\nu$), as well as by Theorems \ref{thm3h} and \ref{bestK} (where
such constraints are absent). As suggested by the title, this work
aims primarily at a precise analysis of conformance to Benford's Law (or
the lack thereof). Correspondingly, much attention is paid to the
special case of $\mu = \B_b$, leading to explicit descriptions of
best (constrained or unconstrained) approximations of the latter
(Corollaries \ref{coL}, \ref{cor4c}, and \ref{coKS}) and the exact
asymptotics of $d_*(\B_b, \nu)$. As indicated earlier, however, the
main results are much more general. To emphasize this fact, two other
simple but illustrative examples of $\mu$ are repeatedly considered as well
(though in less detail than $\B_b$), namely the familiar $\mbox{\tt Beta}(2,1)$
distribution and the (perhaps less familiar) inverse Cantor
distribution. It turns out that while the former is absolutely
continuous (w.r.t.\ Lebesgue measure) and its best approximations
behave like those of $\B_b$ in most respects (Examples \ref{ex37a}, \ref{ex310a},
\ref{ex46a}, and \ref{ex58a}), the latter is
discrete and the behaviour of its best approximations is more delicate
(Examples \ref{ex37b}, \ref{ex310b}, \ref{ex47b}, and \ref{ex58b}). Even with only a few details mentioned, these
examples will help the reader appreciate the versatility of the main results.

The organization of this article is as follows: Section \ref{sec2} reviews
relevant basic
properties of one-dimensional probabilities and the three main
probability metrics used throughout. Each of the Sections \ref{secL}
to \ref{secK} then is devoted specifically to one single metric. In
each section, the problem of best (constrained or
unconstrained) approximation by finitely supported probability
measures is first addressed in complete generality, and then the
results are specialized to $\B_b$ as well as other concrete examples. Section \ref{secC} summarizes and
discusses the quantitative results obtained, and also mentions a few
natural questions for subsequent studies.

\section{Probability metrics}\label{sec2}

Throughout, let $\I\subset \R$ be a compact interval with Lebesgue
measure $\lambda (\I)>0$, and $\cP$ the set of all Borel probability measures on
$\I$. Associate with every $\mu\in \cP$ its {\em distribution
  function\/} $F_{\mu}: \R \to \R,$ given by
$$
F_{\mu}(x) = \mu (\{ y\in \I : y \le x\}) \quad \forall\ x \in \R \, ,
$$
as well as its (upper) {\em quantile function\/}
$F_{\mu}^{-1}:\,
[0,1[ \, \to \R,$ given by
\eqb\label{00}
F_{\mu}^{-1} (x) = \left\{
\begin{array}{lcl}
\min \I & \: & \mbox{\rm if } 0\le x<\mu(\{\min\I\})\, ,\\[1mm]
\sup \{y\in \I : F_{\mu}(y) \le x\} & \: & \mbox{\rm if }\mu(\{\min\I\})\le
x<1 \, .
\end{array}
\right.
\eqe
Note that $F_{\mu}$ and $F_{\mu}^{-1}\!$ both are
non-decreasing, right-continuous, and bounded. The {\em support\/} of $\mu$, denoted
$\mbox{\rm supp}\, \mu$, is the smallest closed subset of $\I$ with $\mu$-measure
$1$. Endowed with the weak topology, the space $\cP$ is compact
and metrizable.

Three important different metrics on $\cP$ are
discussed in detail in this article; for a panorama of other metrics
the reader is referred, e.g., to \cite{GS,RKSF} and the references therein. Given
probabilities $\mu, \nu\in \cP$, their {\em L\'evy distance\/} is
\begin{equation}\label{eq1}
d_{\sf L} (\mu, \nu) = \omega \inf \lt\{ y\ge0: F_{\mu} (\cdot-y) -
y \le F_{\nu} \le F_{\mu} (\cdot+y) + y \rt\}, \,
\end{equation} with $\omega=\max\{1,\lambda(\I)\}/\lambda(\I);$
their $L^r$-{\em Kantorovich\/} (or {\em transport}) {\em distance\/},
with $r\ge 1$, is
\begin{equation}\label{eq2}
d_r (\mu, \nu) = \lambda (\I)^{-1} \left(\int_0^1 \lt|F_{\mu}^{-1}(y) -
F_{\nu}^{-1}(y)\rt|^r {\rm d}y\right)^{1/r} \!\!\! = \lambda (\I)^{-1}
\|F_{\mu}^{-1} - F_{\nu}^{-1}\|_r ;
\end{equation}
and their {\em Kolmogorov\/} (or {\em uniform}) {\em distance\/} is
$$
d_{\sf K} (\mu, \nu) = \sup\nolimits_{x\in \R} \lt|F_{\mu} (x) -
F_{\nu} (x)\rt| = \| F_{\mu} - F_{\nu} \|_{\infty} \, .
$$
Henceforth, the symbol $d_*$ summarily refers to any of
$d_{\sf L}, d_r$, and $d_{\sf K}$. The (unusual) normalizing factors in (\ref{eq1}) and (\ref{eq2}) guarantee that all three
metrics are comparable numerically in that $\sup_{\mu, \nu\in \cP}
d_*(\mu, \nu) =1$ in either case. Note that
$$
d_1(\mu, \nu) = \lambda (\I)^{-1} \int_{\I} \lt|F_{\mu} (x) - F_{\nu} (x)\rt|\, {\rm d}x \quad
\forall\ \mu, \nu \in \cP \, ,
$$
by virtue of Fubini's Theorem. The metrics $d_{\sf L}$ and $d_r$ are equivalent:
They both metrize the weak topology on $\cP$, and hence are separable
and complete. By contrast, the complete metric $d_{\sf K}$ induces a finer topology and is
non-separable. However, when restricted to $\cP_{\sf cts} :=\{\mu \in
\cP : \mu (\{x\}) = 0 \,\enspace \forall\ x \in \I\}$, a dense $G_{\delta}$-set in $\cP$, the metric $d_{\sf K}$ does metrize the weak topology on
$\cP_{\sf cts}$ and is separable. The values of $d_{\sf L}, d_r,$ and $d_{\sf K}$ are not
completely unrelated since, as is easily checked,
\begin{equation}\label{eq3}
d_1 \le\frac{1+\lambda (\I)}{\omega \lambda(\I)} d_{\sf L} \, , \quad
d_r  \le d_s \enspace (\mbox{\rm if } r\le s) \, , \quad  d_1 \le  d_{\sf K} \, , \quad
d_{\sf L} \le \omega d_{\sf K}  \, ,
\end{equation}
and all bounds in (\ref{eq3}) are best possible. Beyond (\ref{eq3}), however, no relative bounds exist between $d_{\sf
  L}, d_r$, and $d_{\sf K}$ in general: If $\ast \ne 1$, $\ast\ne
\circ $, and $(\ast, \circ ) \not \in \{  ({\sf L}, {\sf K}), (r,s)
\}$ with $r\le s$ then
$$
\sup\nolimits_{\mu, \nu\in \cP : \mu \ne \nu} \frac{d_{\ast} (\mu,
  \nu)}{d_{\circ} (\mu, \nu)} = +\infty \, .
$$
Each metric
$d_{\ast}$, therefore, captures a different aspect of $\cP$ and
deserves to be studied independently. To illustrate this further, let
$\I = [0,1]$, $\mu=\delta_0 \in \cP$, and
$\mu_k=\lt(1-k^{-1}\rt)\delta_0+k^{-1}\delta_{k^{-2}}$ for $k\in\N$;
here and throughout, $\delta_a$ denotes the Dirac (probability)
measure concentrated at $a\in\mathbb{R}.$ Then $\lim_{k\to \infty} d_{\ast} (\mu, \mu_k) = 0$, but the
rate of convergence differs between metrics:
$$
d_{\sf L} (\mu, \mu_k) = k^{-2} \, , \quad
d_r (\mu, \mu_k) =k^{-2-1/r}\, , \quad
d_{\sf K} (\mu, \mu_k) = k^{-1} \, \quad \forall\ k\in\N.
$$

The goal of this article is first to identify, for each metric $d_*$
introduced earlier, the best finitely supported $d_*$-approximation(s)
of any given $\mu \in \cP$. The general results are then applied
to Benford's Law, as well as other concrete examples. Specifically, if
$\mu=\B_b$ for some $b>1$ then it is automatically assumed that $\I=[1,b].$ The following unified notation and
terminology is used throughout: For every $n\in \N$, let $\Xi_n=\{x\in\I^n: x_{,1}\le\ldots\le x_{,n}\}$,
$\Pi_n=\{p\in\R^n: p_{,j}\ge0,\ \sum_{j=1}^np_{,j}=1\}$, and for
each $x\in \Xi_n$ and $p\in \Pi_n$ define
$\delta_x^p=\sum_{j=1}^np_{,j}\delta_{x_{,j}}$. For convenience,
$x_{,0}:=-\infty$ and $x_{,n+1}:=+\infty$ for every $x\in\Xi_n,$ as well as $P_{,i}=\sum_{j=1}^ip_{,j}$ for $i=0,\ldots,n$ and $p\in\Pi_n$; note that $P_{,0}=0$ and $P_{,n}=1.$
Henceforth, usage of the symbol $\delta_x^p$ tacitly assumes that $x\in\Xi_n$ and $p\in\Pi_n,$
for some $n\in\mathbb{N}$ either specified explicitly or else clear from the context.
Call $\delta_x^p$ a {\em best
  $d_*$-approximation of $\mu\in \cP$, given $x\in \Xi_n$}  if
$$
d_*\lt(\mu,\delta_x^p\rt)\le d_*\lt(\mu,\delta_x^q\rt)\quad  \forall\
q\in\Pi_n\, .
$$
Similarly, call $\delta_x^p$ a  {\em best
  $d_*$-approximation of $\mu$, given $p\in \Pi_n$} if
$$
d_*\lt(\mu,\delta_x^p\rt)\le d_*\lt(\mu,\delta_y^p\rt)\quad  \forall\
y\in\Xi_n\, .
$$
Denote by $\delta_x^{\bullet}$ and $\delta_{\bullet}^p$ any best
$d_*$-approximation of $\mu$, given $x$ and $p$, respectively. Best
$d_*$-approximations, given $p=u_n = (n^{-1}, \ldots , n^{-1})$ are referred to
as best {\em uniform\/} $d_*$-approxima\-ti\-ons, and denoted
$\delta_{\bullet}^{u_n}$. Finally, $\delta_x^p$ is a {\em best
  $d_*$-approximation of\/} $\mu \in \cP$, denoted
$\delta_{\bullet}^{\bullet, n}$, if
$$
d_*\lt(\mu,\delta_x^p\rt)\le d_* \lt(\mu,\delta_y^q\rt)\quad  \forall\
y\in \Xi_n, q\in\Pi_n\, .
$$ Notice that usage of the symbols $\delta_x^{\bullet},$
$\delta^p_{\bullet},$ and $\delta_{\bullet}^{\bullet,n}$ always refers
to a specific metric $d_*$ and probability measure $\mu\in\cP$, both usually clear from the context.

Information theory sometimes refers to $d_{\ast}
\lt(\mu,\delta_{\bullet}^{\bullet, n}\rt)$ as the {\em $n$-th quantization
error}, and to $\lim_{n\to\infty} n d_*\lt(\mu,\delta_{\bullet}^{\bullet, n}
\rt)$, if it exists, as the {\em quantization coefficient\/}  of $\mu$; see, e.g., \cite{GL}. By analogy,
$d_{\ast} (\mu,\delta_{\bullet}^{u_n})$ and $\lim_{n\to\infty} n
d_{\ast} (\mu,\delta_{\bullet}^{u_n})$, if it exists, may be called the {\em $n$-th 
uniform quantization error\/} and the {\em uniform quantization
coefficient}, respectively.

\section{L\'evy approximations}\label{secL}

This section identifies best finitely supported $d_{\sf
  L}$-approximations (constrained or unconstrained) of a given
$\mu\in\cP$. To do this in a transparent way, it is helpful to first 
consider more generally a few elementary properties of non-decreasing
functions. These properties are subsequently specialized to either $F_{\mu}$ or $F_{\mu}^{-1}$.

Throughout, let $f:\R \to \overline{\R}$ be non-decreasing,
and define $f(\pm \infty) = \lim_{x\to \pm \infty} f(x)\in \overline{\R}$, where
$\overline{\R} = \R \cup \{ -\infty, + \infty\}$ denotes the extended
real line with the usual order and topology. Associate with $f$ two
non-decreasing functions $f_{\pm}: \R \to \overline{\R}$, defined as
$f_{\pm} (x) = \lim_{\varepsilon \downarrow 0} f(x \pm
\varepsilon)$. Clearly, $f_-$ is left-continuous whereas $f_+$ is
right-continuous, with $f_{\pm}(-\infty) = f(-\infty)$,
$f_{\pm}(+\infty) = f(+\infty)$, as well as $f_- \le f \le f_+$, and
$f_+(x) \le f_- (y)$ whenever $x<y$; in particular, $f_-(x) =
f_+ (x)$ if and only if $f$ is continuous at $x$. The (upper) {\em
  inverse function\/} $f^{-1}: \R \to \overline{\R}$ is given by
$$
f^{-1} (t) = \sup \{ x\in \R : f(x) \le t\} \quad \forall\ t \in \R \, ;
$$
by convention, $\sup \varnothing := -\infty$ (and $\inf
\varnothing:= +\infty$). Note that (\ref{00}) is consistent with this
notation. For what follows, it is useful to recall a few basic
properties of inverse functions; see, e.g., \cite[Sec.3]{XB} for details.

\begin{prp}\label{prop3a}
Let $f:\R \to \overline{\R}$ be non-decreasing. Then $f^{-1}$ is
non-decreas\-ing and right-continuous. Also, $(f_{\pm})^{-1} = f^{-1}$,
and $(f^{-1})^{-1} = f_+$.
\end{prp}
Given two non-decreasing functions $f,g: \R \to \overline{\R}$, by a slight abuse of notation, and inspired by (\ref{eq1}), let
$$
d_{\sf L} (f,g) = \inf \{ y \ge 0 : f(\, \cdot \, - y) - y \le g \le
f(\, \cdot \, + y) + y\}\in[0,+\infty] \, .
$$
For instance, $d_{\sf L} (\mu,\nu) =
\omega d_{\sf L} (F_{\mu}, F_{\nu})$ for all $\mu,\nu\in \cP$. It is readily checked that
$d_{\sf L}$ is symmetric, satisfies the triangle inequality, and $d_{\sf
  L}(f,g)>0$ unless $f_- = g_-$, or equivalently, $f_+ =
g_+$. Crucially, the quantity $d_{\sf L}$ is invariant under
inversion.

\begin{prp}\label{prop3b}
Let $f,g:\R \to \overline{\R}$ be non-decreasing. Then $d_{\sf
  L} (f^{-1}, g^{-1}) = d_{\sf L} (f,g)$.
\end{prp}
Thus, for instance, $d_{\sf
  L} (\mu, \nu) = \omega d_{\sf L} (F_{\mu}^{-1}, F_{\nu}^{-1})$  for all
$\mu, \nu\in \cP$. In general, the value of $d_{\sf L}(f,g)$ may equal
$+\infty$. However, if the set $\{f\neq g\}:=\{x\in\R: f(x)\neq
g(x)\}$ is bounded then $d_{\sf L}(f,g)<+\infty.$ Specifically, notice 
that $\lt\{ F_{\mu}\neq F_{\nu}\rt\}\subset\I$ and $\{F_{\mu}^{-1}\neq F_{\nu}^{-1}\}\subset[0,1[$
both are bounded for all $\mu,\nu\in\cP$.

Given a non-decreasing function $f:\R \to \overline{\R}$, let
$I\subset \overline{\R}$ be any interval with the property that
\begin{equation}\label{eq3100}
f_- (\sup I), -f_+ (\inf I) < + \infty \, ,
\end{equation}
and define an auxiliary function $\ell_{f,I}:\R \to \R$ as
$$
\ell_{f,I} (x) = \inf \{ y\ge 0 : f_- (\sup I - y) - y \le x \le f_+
(\inf I + y) + y\} \, .
$$
Note that for each $x\in \R$, the set on the right equals
$[a,+\infty[$ with the appropriate $a \ge 0$, and hence simply
$\ell_{f,I} (x) = a$. Clearly, $\ell_{f,J} \le \ell_{f,I}$
whenever $J\subset I$. Also, for every $a \in \R$,
the function $\ell_{f,\{a \}}$ is non-increasing on $]-\infty, f_-(a)]$,
vanishes on $[f_-(a), f_+(a)]$, and is non-decreasing on $[f_+(a),
+\infty[$. A few elementary properties of $\ell_{f,I}$ are
straightforward to check; they are used below to establish the main
results of this section.

\begin{prp}\label{prop3c}
Let $f:\R \to \overline{\R}$ be non-decreasing, and $I\subset \overline{\R}$ an interval satisfying {\rm (\ref{eq3100})}. Then
$\ell_{f,I}$ is Lipschitz continuous, and
$$
0 \le \ell_{f,I} (x) \le |x|  + \max \{0, f_-(\sup I), - f_+ (\inf I)\}
\quad \forall\ x \in \R \, .
$$
Moreover, $\ell_{f,I}$ attains a minimal value
$$
\ell_{f,I}^* := \min \nolimits_{x\in \R} \ell_{f,I} (x) = \min \{ y\ge
0: f_-(\sup I - y) - y \le f_+(\inf I + y) + y\} \: \ge \: 0 
$$
which is positive unless $f_-(\sup I) \le f_+ (\inf I)$.
\end{prp}
For $\mu\in \cP$, note that (\ref{eq3100}) automatically holds
if $f= F_{\mu}$, or if $f=F_{\mu}^{-1}$ and $I\subset [0,1]$. In these cases, therefore, $\ell_{f,I}$ has the
properties stated in Proposition \ref{prop3c}, and $\ell_{f,I}^*\le \frac12$.

When formulating the main results, the following quantities are
useful: Given $\mu\in \cP$, $n\in \N$, and $x\in \Xi_n$, let
$$
{\sf L}^{\bullet}(x) = \max \lt\{
\ell_{F_{\mu}, [-\infty, x_{,1}]}(0), \ell_{F_{\mu}, [x_{,1},
  x_{,2}]}^*, \ldots,  \ell_{F_{\mu}, [x_{,n-1},
  x_{,n}]}^*, \ell_{F_{\mu}, [x_{,n}, +\infty]}(1)
\rt\} \, ;
$$
similarly, given $p\in \Pi_n$, let
$$
{\sf L}_{\bullet} (p) = \max\nolimits_{j=1}^n
\ell_{F_{\mu}^{-1}, [P_{,j-1}, P_{,j}]}^* \, .
$$
To illustrate these quantities for a concrete example, consider $\mu = \B_b$, where $\ell_{F_{\mu},[x_{,j}, x_{,j+1}]}^*$ is
the unique solution of
$$
b^{2\ell} = \frac{x_{,j+1} - \ell}{x_{,j} + \ell} \quad j=1, \ldots ,
n-1 \, ,
$$
whereas $\ell_{F_{\mu}, [-\infty, x_{,1}]}(0)$ and $\ell_{F_{\mu},
  [x_{,n}, +\infty]}(1)$ solve $b^{\ell} = x_{,1} -\ell$ and $b^{\ell}
= b/(x_{,n} + \ell)$, respectively. (Recall that $1\le
x_{,1} \le \ldots  \le x_{,n} \le b$.) Similarly, $\ell_{F_{\mu}^{-1}, [P_{,j-1},
  P_{,j}]}^*$ is the unique solution of
$$
2\ell  = b^{P_{,j} - \ell} - b^{P_{,j-1} + \ell} \quad j = 1, \ldots ,
n \, ;
$$
in particular, $j \mapsto \ell_{F_{\mu}^{-1}, [(j-1)/n, j/n]}^*$ is
increasing, and hence $ {\sf L}_{\bullet}(u_n)$ is the unique
solution of
\begin{equation}\label{eq3101}
2L = b^{1 - L} - b^{1 + L - 1/n} \, .
\end{equation}
By using functions of the form $\ell_{f,I}$, the value of $d_{\sf L}
(\mu, \nu)$ can easily be computed whenever $\nu$ has finite
support.
\begin{lem}\label{lem3d}
Let $\mu \in \cP$ and $n\in \N$. For every $x\in \Xi_n$ and $p\in
\Pi_n$,
\begin{equation}\label{eq3102}
d_{\sf L} \lt(\mu, \delta_x^p\rt) = \omega \max \nolimits_{j=0}^n
\ell_{F_{\mu} , [x_{,j}, x_{,j+1}]} (P_{,j}) = \omega
\max\nolimits_{j=1}^n \ell_{F_{\mu}^{-1}, [P_{,j-1}, P_{,j}]}(x_{,j})
\, .
\end{equation}
\end{lem}

\begin{proof}
Label $x\in \Xi_n$ uniquely as
\begin{align*}
x_{,j_0 + 1} = \ldots = x_{,j_1} & < x_{,j_1+1} = \ldots = x_{,j_2}< x_{,j_2 +
1} = \ldots \\ & < \ldots = x_{,j_{m-1}}< x_{,j_{m-1}+1} = \ldots = x_{,j_m}
\, ,
\end{align*}
with integers $i\le j_i \le n$ for $1\le i\le m$, and $j_0 = 0$, $j_m
= n$, and define $y \in \Xi_m$ and $ q\in \Pi_m$ as
$y_{,i} = x_{,j_i}$ and $q_{,i} = P_{,j_i} - P_{,
  j_{i-1}}$, respectively, for $i=1, \ldots , m$. For convenience, let $I_j =
[x_{,j}, x_{, j+1}]$ for $j=0, \ldots, n$, and $J_i =
[y_{,i}, y_{,i+1}]= I_{j_i}$ for $i=0, \ldots,
m$. With this, $\delta_{y}^{q} = \delta_x^p$,
and
\begin{align*}
\omega^{-1} & d_{\sf L} \lt(\mu, \delta_x^p\rt)  = d_{\sf L} (F_{\mu},
F_{\delta_{y}^{q}}) \\
&= \inf \{ t \ge 0 : F_{\mu-} (y_{,i+1} - t) - t \le
Q_{,i} \le F_{\mu} (y_{,i} +t) + t \quad
\forall\ i=0, \ldots, m\} \\
& = \max\nolimits_{i=0}^m \ell_{F_{\mu}, J_i} (Q_{,i}) \\ &  \le  
\max\nolimits_{j=0}^n \ell_{F_{\mu}, I_j} (P_{,j}) \, .
\end{align*}
To prove the reverse inequality, pick any $j=0,\ldots, n$. If $x_{,j}<
x_{,j+1}$ then $I_j = J_i$ and $P_{,j} =
Q_{,i}$, with the appropriate $i$, and hence
$\ell_{F_{\mu}, I_j} (P_{,j}) = \ell_{F_{\mu},
  J_i}(Q_{,i})$. If $x_{,j} = x_{,j+1}$ then
$I_j = \{y_{,i}\}$ for some $i$. In this case either
$P_{,j} < F_{\mu-} (y_{,i})$ and $Q_{,i-1}\le
P_{,j}$, and hence
$$
\ell_{F_{\mu}, I_j} (P_{,j}) = \ell_{F_{\mu}, \{y_{,i}\}}
(P_{,j}) \le \ell_{F_{\mu}, \{y_{,i}\}}
(Q_{,i-1}) \le \ell_{F_{\mu}, J_{i-1}}
(Q_{,i-1}) \, ;
$$
or $F_{\mu-} (y_{,i}) \le P_{,j} \le F_{\mu}
(y_{,i})$, and hence $\ell_{F_{\mu}, I_j} (P_{,j}) = \ell_{F_{\mu}, \{y_{,i}\}}
(P_{,j}) = 0$;
or $P_{,j} > F_{\mu}(y_{,i})$ and $Q_{,i} \ge
P_{,j}$, and hence
$$
\ell_{F_{\mu}, I_j} (P_{,j}) = \ell_{F_{\mu}, \{y_{,i}\}}
(P_{,j}) \le \ell_{F_{\mu}, \{y_{,i}\}}
(Q_{,i}) \le \ell_{F_{\mu}, J_i}
(Q_{,i}) \, .
$$
In all three cases, therefore, $\omega ^{-1} d_{\sf L} \lt(\mu ,\delta_x^p\rt) \ge
\max_{j=0}^n \ell_{F_{\mu}, I_j} (P_{,j})$, which establishes the
first equality in (\ref{eq3102}). The second equality, a consequence
of Proposition \ref{prop3b}, is proved analogously.
\end{proof}

Utilizing Lemma \ref{lem3d}, it is straightforward to characterize the
best finitely supported $d_{\sf L}$-approximations of $\mu\in \cP$ with prescribed locations.

\begin{theorem}\label{thm3e}
Let $\mu \in \cP$ and $n\in \N$. For every $x\in \Xi_n$, there exists
a best $d_{\sf L}$-approximation of $\mu$, given $x$. Moreover,
$d_{\sf L} \lt(\mu, \delta_x^p \rt) = d_{\sf L} \bigl(\mu, \delta_x^{\bullet} \bigr)$ if
and only if, for every $j=0,\ldots, n$,
\begin{equation}\label{eq3103}
x_{,j} < x_{,j+1} \enspace \Longrightarrow \enspace
\ell_{F_{\mu}, [x_{,j}, x_{,j+1}]} (P_{,j}) \le {\sf L}^{\bullet} (x)\, ,
\end{equation}
and in this case $d_{\sf L} \lt(\mu, \delta_x^p \rt)= \omega {\sf L}^{\bullet}(x)$.
\end{theorem}

\begin{proof}
Fix $\mu \in \cP$, $n\in \N$, and $x\in \Xi_n$. As in the proof of
Lemma \ref{lem3d}, write $I_j=[x_{,j}, x_{,j+1}]$ for convenience. By
(\ref{eq3102}), for every $p\in \Pi_n$,
\begin{align*}
d_{\sf L} \lt(\mu, \delta_x^p \rt) & = \omega \max\nolimits_{j=0}^n
\ell_{F_{\mu}, I_j} (P_{,j}) \\
& \ge \omega \max \{ \ell_{F_{\mu}, I_0}(0), \ell_{F_{\mu}, I_1}^*
, \ldots , \ell_{F_{\mu}, I_{n-1}}^*, \ell_{F_{\mu}, I_n} (1)\} =
\omega {\sf L}^{\bullet} (x) \, .
\end{align*}
As seen in the proof of Lemma \ref{lem3d}, validity of \eqref{eq3103}
implies $\ell_{F_{\mu},[x_{,j},x_{,j+1}]}(P_{,j})$ $\le{\sf L}^{\bullet}(x)$ for {\em all} $j=0,\ldots,n$. Thus $\delta_x^p$ is a best $d_{\sf L}$-approximation of $\mu$, given
$x$, whenever (\ref{eq3103}) holds, i.e., the latter is sufficient for
optimality. On the other hand, consider $q\in \Pi_n$ with
$$
Q_{,j} = \frac12 \Bigl( F_{\mu -} \bigl(x_{,j+1} - {\sf
  L}^{\bullet}(x )\bigr) + F_{\mu} \bigl(x_{,j} + {\sf L}^{\bullet}(x)
\bigr)\Bigr)
\quad \forall j = 1, \ldots , n-1 \, .
$$
Note that $q$ is well-defined, since $j\mapsto
Q_{,j}$ is non-decreasing, and $0\le Q_{,j}\le
1$ for all $j=1, \ldots , n-1$. Moreover, by the definition of ${\sf
  L}^{\bullet}(x)$,
$$
\ell_{F_{\mu}, I_j} \lt(Q_{,j}\rt) \le {\sf L}^{\bullet}(x)  \quad
\forall j=0, \ldots , n \, ,
$$
and hence $d_{\sf L} \lt(\delta_x^{q}, \mu\rt) = \omega
{\sf L}^{\bullet}(x)$. This shows that best $d_{\sf
  L}$-approximations of $\mu$, given $x$, do exist, and
(\ref{eq3103}) also is necessary for optimality. 
\end{proof}

Best finitely supported $d_{\sf L}$-approximations of any $\mu\in \cP$ with prescribed
weights can be characterized in a similar manner. By virtue of (\ref{eq3102}), the proof
of the following is completely analogous to the proof of Theorem
\ref{thm3e} above.

\begin{prp}\label{prop3f}
Let $\mu \in \cP$ and $n\in \N$. For every $p\in \Pi_n$, there exists
a best $d_{\sf L}$-approximation of $\mu$, given $p$. Moreover,
$d_{\sf L} \lt(\mu, \delta_x^p \rt) = d_{\sf L} \lt(\mu, \delta_{\bullet}^p \rt)$ if
and only if, for every $j=1,\ldots, n$,
\begin{equation}\label{eq3104}
P_{,j-1} < P_{,j} \enspace \Longrightarrow \enspace
\ell_{F_{\mu}^{-1}, \lt[P_{,j-1}, P_{,j}\rt]} (x_{,j}) \le {\sf L}_{\bullet} (p),
\end{equation}
and in this case $d_{\sf L} \lt(\mu, \delta_x^p\rt)= \omega {\sf L}_{\bullet}(p)$.
\end{prp}

\rb
(i) With $f, I$ as in Proposition \ref{prop3c}, for every $a \in
\R$ the set $\{\ell_{f,I} \le
a\}$ is a (possibly empty or one-point) interval. Thus, conditions (\ref{eq3103}) and (\ref{eq3104}) are very similar in
spirit to the requirements of \cite[Thm.5.1 and 5.5]{XB}, restated in Proposition~\ref{prop4a} below, though
the latter may be quite a bit easier to work with in concrete calculations.

(ii) Note that if $n=1$ then (\ref{eq3103}) holds automatically, whereas
(\ref{eq3104}) shows that $d_{\sf L} (\mu, \delta_{a})$ is
minimal precisely if the function $\ell_{F_{\mu}^{-1},[0,1]}$ attains its
minimal value at $a$.
\re

As a corollary, Proposition \ref{prop3f} identifies all best {\em uniform\/}
$d_{\sf L}$-approximations of $\beta_b$ with $b>1$. Recall that
$\I=[1,b]$, and hence $\omega=\displaystyle \frac{\max\{b,2\}-1}{b-1}=:\omega_b$ in this case.

\begin{cor}\label{coLU}
Let $b>1$ and $n\in \N$. Then $\delta_x^{u_n}$ is a
best uniform $d_{\sf L}$-appro\-xi\-ma\-tion of\, $\B_b$ if and only if
$$
b^{j/n -L} - L \le x_{,j} \le b^{(j-1)/n +L} +
L\quad \forall j = 1, \ldots , n \, ,
$$
where $L$ is the unique solution of
{\rm (\ref{eq3101})}; in particular, $\# \mbox{\rm supp}\,
\delta_{\bullet}^{u_n}= n$. Moreover, $d_{\sf L} \lt(\B_b ,
\delta_{\bullet}^{u_n}\rt)=\omega_b L$, and
$$
\lim\nolimits_{n\to \infty} n d_{\sf L} \lt(\B_b, \delta_{\bullet}^{u_n}\rt)
= \frac{\max\{b,2\}-1}{2b-2}\cdot\frac{b \log b}{1+ b \log b} \, .
$$
\end{cor}

\begin{example}\label{ex37a}
Consider the $\mbox{\tt Beta}(2,1)$ distribution
on $\I = [0,1]$, i.e., let $F_{\mu}(x)=x^2$ for all $x\in
\I$. Given $n\in \N$, it is straightforward to check that, analogously
to (\ref{eq3101}), ${\sf L}_{\bullet}(u_n)$ is the unique solution of
\begin{equation}\label{eq37a1}
L \sqrt{\frac2{n} - 4L^2} = \frac1{2n} - L \, ,
\end{equation}
and $\delta_x^{u_n}$ with $x\in \Xi_n$ is a best uniform $d_{\sf
  L}$-approximation of $\mu$ if and only if
$$
\sqrt{\frac{j}{n} - L} - L \le x_{,j} \le \sqrt{\frac{j-1}{n} + L} + L
\quad \forall j=1 , \ldots , n \, .
$$
Moreover, $d_{\sf L} (\mu , \delta_{\bullet}^{u_n}) = L$, and
(\ref{eq37a1}) yields that $\lim_{n\to \infty} n d_{\sf L} (\mu ,
\delta_{\bullet}^{u_n}) = \frac12$. Unlike in the case of $\beta_b$,
it is possible to have $\# \mbox{\rm supp}\, \delta_{\bullet}^{u_n}
<n$ whenever $n\ge 10$.
\end{example}

\begin{example}\label{ex37b}
Let again $\I = [0,1]$ and consider $\mu \in \cP$ with $\mu (\{
i2^{-m}\}) = 3^{-m}$ for every $m\in \N$ and every odd $1\le i <
2^m$. Thus $\mu$ is a discrete measure with $\mbox{\rm supp}\, \mu =
\I$. In fact, $\mu$ simply is the {\em inverse Cantor distribution},
in the sense that $F_{\mu}^{-1} (x) = F_{\nu}(x)$ for all $x\in \I$,
where $\nu$ is the $\log 2/\log 3$-dimensional Hausdorff measure on
the classical Cantor middle-thirds set. Given $n\in \N$, Proposition
\ref{prop3f} guarantees the existence of a best uniform $d_{\sf
  L}$-approximation of $\mu$, though the explicit value of ${\sf
  L}_{\bullet}(u_n)$ is somewhat cumbersome to determine. Still,
utilizing the self-similarity of $F_{\mu}^{-1}$, one finds that
\begin{equation}\label{eq36n1}
\frac{1}{216}\le\liminf\nolimits_{n\to\infty}nd_{\sf
  L}\lt(\mu,\delta_{\bullet}^{u_n}\rt)\le\frac{1}{3}\, ,
\quad \limsup\nolimits_{n\to\infty}nd_{\sf
  L}\lt(\mu,\delta_{\bullet}^{u_n}\rt)=\frac{1}{2} \, .
\end{equation}
Thus $(n^{-1})$ is the precise rate of decay of $\bigl( d_{\sf L}
(\mu, \delta_{\bullet}^{u_n})\bigr)$, just as in the case of $\B_b$
and $\mbox{\tt Beta}(2,1)$, but unlike for the latter, $\lim_{n\to
  \infty} n d_{\sf L} (\mu, \delta_{\bullet}^{u_n})$ does not exist.
\end{example}

By combining Theorem \ref{thm3e} and Proposition \ref{prop3f}, it is possible to
characterize the best $d_{\sf L}$-approximations of $\mu \in \cP$ as
well, that is, to identify the minimizers of $\nu \mapsto d_{\sf L}(\mu, \nu)$
subject only to the requirement that $\# \mbox{\rm supp}\, \nu \le n$. To
this end, associate with every non-decreasing
function $f:\R \to \overline{\R}$ and every number $a\ge 0$ a
map $T_{f,a}:\overline{\R} \to \overline{\R}$, according to
$$
T_{f, a} (x) = f_+ \lt(f^{-1}(x+a) + 2a\rt) +a \quad
\forall\ x \in \overline{\R} \, .
$$
For every $n\in \N$, denote by $T_{f,a}^{[n]}$ the $n$-fold
composition of $T_{f,a}$ with itself. The following properties of
$T_{f,a}$ are readily verified.

\begin{prp}\label{prop3g}
Let $f:\R \to \overline{\R}$ be non-decreasing, $a \ge 0$, and
$n\in \N$. Then $T_{f,a}^{[n]}$ is non-decreasing and
right-continuous. Also, $a \mapsto
T_{f,a}^{[n]}(x)$ is increasing and right-continuous for every
$x\in \R$, and if
$x\le a + f(+\infty)$ then the sequence $\lt(
T_{f,a}^{[k]}(x)\rt)$ is non-decreasing.
\end{prp}

To utilize Proposition \ref{prop3g} for the $d_{\sf L}$-approximation
problem, let $f=F_{\mu}$ with $\mu \in \cP$. Then  $\lt(T_{F_{\mu},
  a}^{[k]}(0)\rt)$ is non-decreasing; in fact, $\lim_{k\to
  \infty} T_{F_{\mu},a}^{[k]}(0) = a + 1$. On the other hand, given $n\in \N$, clearly
$T_{F_{\mu}, a}^{[n]}(0)\ge 1$ for all
$a\ge1$, and hence
$$
{\sf L}_{\bullet}^{\bullet, n} := \min \lt\{a \ge 0 : T_{F_{\mu},
  a}^{[n]}(0) \ge 1\rt\} < +\infty \, .
$$
Note that ${\sf L}_{\bullet}^{\bullet, n}$ only depends on $\mu$ and
$n$. The sequence $\lt({\sf L}_{\bullet}^{\bullet, n}\rt)$ is
non-increasing, and $n{\sf L}_{\bullet}^{\bullet,n}\le \frac12$ for every
$n$. Also, ${\sf L}_{\bullet}^{\bullet, n}=0$ if and only if $\#
\mbox{\rm supp}\, \mu \le n$.

For a concrete example, consider $\mu = \B_b$ with $a < \frac12 (b-1)$, where
$$
T_{F_{\mu}, a}(x) = \left\{
\begin{array}{ll}
a & \mbox{\rm if } x< - a \, , \\
a + \log_b (b^{x+a} + 2a) & \mbox{\rm if } -a \le
x < -a + \log_b(b-2a) \, , \\
a + 1 & \mbox{\rm if } x\ge - a + \log_b (b-2a) \, ,
\end{array}
\right.
$$
from which it is easily deduced that ${\sf L}_{\bullet}^{\bullet, n}$
is the unique solution of
\begin{equation}\label{eq3105}
b^{2n L } = \frac{2L  + b (b^L - b^{-L})}{2 L + b^L - b^{- L}} \, .
\end{equation}
As the following result shows, the quantity ${\sf
  L}^{\bullet,n}_{\bullet}$ always plays a central role in identifying
best (unconstrained) $d_{\sf L}$-approximations of a given $\mu \in \cP$.

\begin{theorem}\label{thm3h}
Let $\mu \in \cP$ and $n\in \N$. There exists a best $d_{\sf
  L}$-approximation of $\mu$, and $d_{\sf L} \lt(\mu, \delta_{\bullet}^{\bullet, n}\rt) =
\omega {\sf L}_{\bullet}^{\bullet, n}$. Moreover, for every $x\in \Xi_n$ and $p\in \Pi_n$,
the following are equivalent:
\noindent\begin{enumerate}
\item[{\rm(i)}] $d_{\sf L} \lt(\mu, \delta_x^p \rt) = d_{\sf L}
  \lt(\mu, \delta_{\bullet}^{\bullet, n}\rt)$;
\item[{\rm (ii)}] all implications in {\rm (\ref{eq3103})} are valid with ${\sf
    L}^{\bullet}(x)$ replaced by ${\sf L}^{\bullet, n}_{\bullet}$;
\item[{\rm (iii)}] all implications in {\rm (\ref{eq3104})} are valid with ${\sf
    L}_{\bullet}(p)$ replaced by ${\sf L}^{\bullet, n}_{\bullet}$.
\end{enumerate}
\end{theorem}

\begin{proof}
To see that best $d_{\sf L}$-approximations of $\mu$ do exist, simply
note that the set $\{\nu \in \cP: \# \mbox{\rm supp}\, \nu \le n\}$ is
compact, and the function $\nu \mapsto d_{\sf L} (\mu, \nu)$ is
continuous, hence attains a minimal value for some $\nu = \delta_x^p$
with $x\in \Xi_n$ and $p\in \Pi_n$. Clearly, any such $\delta_x^p$ also is a best
approximation of $\mu$, given $p$. By Proposition \ref{prop3f},
therefore, $d_{\sf L}(\mu, \delta_x^p) = \omega {\sf
  L}_{\bullet}(p)$, as well as
$$
F_{\mu-}^{-1} \bigl( P_{,j} - {\sf L}_{\bullet}(p)\bigr) -  {\sf
  L}_{\bullet}(p) \le x_{,j} \le F_{\mu}^{-1} \bigl( P_{,j-1} +  {\sf
  L}_{\bullet}(p)\bigr) +  {\sf L}_{\bullet}(p)
$$ whenever $P_{,j-1}<P_{,j}$, and indeed for every $j=1,\ldots , n$.
It follows that $P_{,j} \le T_{F_{\mu},  {\sf  L}_{\bullet}(p)}(P_{,j-1})$ for all $j$, and hence $1= P_{,n} \le
T_{F_{\mu},  {\sf  L}_{\bullet}(p)}^{[n]}(0)$, that is, ${\sf
    L}_{\bullet}^{\bullet, n} \le {\sf
    L}_{\bullet}(p)$. This shows that $d_{\sf L} (\mu, \delta_x^p) \ge
  \omega {\sf L}_{\bullet}^{\bullet, n}$. To establish the
  reverse inequality, let
$$
m = \min \left\{ i\ge 1 :  T_{F_{\mu},  {\sf  L}_{\bullet}^{\bullet, n}}^{[i]}(0) \ge 1\right\} \, .
$$
Clearly, $1\le m \le n$, and $ {\sf  L}_{\bullet}^{\bullet, m} =  {\sf
  L}_{\bullet}^{\bullet, n}$. Define $q\in \Pi_m$ via
$$
Q_{,i} =  T_{F_{\mu},  {\sf  L}_{\bullet}^{\bullet, n}}^{[i]}(0)
\quad \forall\ i =1, \ldots , m-1 \, .
$$
Note that $i\mapsto Q_{,i}$ is non-decreasing, and $0\le Q_{,i}\le1$,
so $q$ is well-defined. Also, consider $y\in \Xi_m$ with
$$
y_{,i} = \frac12 \bigl(  F_{\mu-}^{-1} (Q_{,i} -{\sf
  L}_{\bullet}^{\bullet, m} ) + F_{\mu}^{-1} (Q_{,i-1} + {\sf
  L}_{\bullet}^{\bullet, m})\bigr) \quad \forall\ i=1, \ldots , m \, .
$$
By the definitions of ${\sf L}_{\bullet}^{\bullet, m}$, $q$, and $y$,
$$
\ell_{F_{\mu}^{-1}, [Q_{,i-1}, Q_{,i}]} (y_{,i}) \le {\sf
  L}_{\bullet}^{\bullet, m} \quad \forall i = 1, \ldots , m \, ,
$$
and hence
$$
d_{\sf L} \lt(\mu, \delta_x^p\rt) \le d_{\sf L} \lt(\mu, \delta_y^q\rt) =
\omega\max\nolimits_{i=1}^n \ell_{F_{\mu}^{-1}, [Q_{,i-1}, Q_{,i}]} (y_{,i}) \le \omega {\sf
  L}_{\bullet}^{\bullet, m} =  \omega {\sf
  L}_{\bullet}^{\bullet, n} \, .
$$
This shows that indeed $d_{\sf L}\lt(\mu, \delta_x^p\rt) = \omega {\sf
  L}_{\bullet}^{\bullet, n}$ and also proves
(i)$\Rightarrow$(iii). The implication (i)$\Rightarrow$(ii) follows by
a similar argument. That, conversely, either of (ii) and (iii) implies
(i) is evident from (\ref{eq3102}), together with the fact that, as
seen in the proof of Lemma~\ref{lem3d} above, validity of \eqref{eq3103} and
\eqref{eq3104} implies $\max_{j=0}^n\ell_{F_{\mu}, [x_{,j}, x_{,j+1}]}
(P_{,j}) \le {\sf L}^{\bullet} (x)$ and
$\max_{j=1}^n\ell_{F_{\mu}^{-1}, \lt[P_{,j-1}, P_{,j}\rt]} (x_{,j})
\le {\sf L}_{\bullet} (p)$, respectively. 
\end{proof}

\rb
(i) The above proof of Theorem \ref{thm3h} shows that in fact
$$
{\sf L}_{\bullet}^{\bullet, n} = \min\nolimits_{x\in \Xi_n} {\sf
  L}^{\bullet} (x) =  \min\nolimits_{p\in \Pi_n} {\sf
  L}_{\bullet} (p) \, .
$$
(ii) Theorem \ref{thm3h} is similar to classical one-dimensional quantization results
as presented, e.g., in \cite[Sec.5.2]{GL}. What makes the theorem (and
its analogue, Theorem~\ref{bestK} in Section 5) particularly appealing
is that its conditions (ii) and (iii) not only are necessary for
optimality, but also sufficient. By contrast, it is well known that
sufficient conditions for best $d_*$-approximations may be hard to come
by in general; see, e.g., \cite[Sec.4.1]{GL}, and also Proposition~\ref{prop4a}(iii) below,
regarding the case of $* =1$.
\re

When specialized to $\mu = \B_b$, Theorem \ref{thm3h} yields the best
finitely supported $d_{\sf L}$-approxima\-ti\-ons of Benford's Law.

\begin{cor}\label{coL}
Let $b>1$ and $n\in \N$. Then the best $d_{\sf L}$-approximation of\, $\B_b$
is $\delta_x^p$, with
\begin{align*}
x_{,j}& = b^{(2j-1)L} + 2L \frac{b^{2jL}-1}{b^{2L} -1} - L = b^{P_{,j} - L} - L \, ,\\
P_{,j} & = \frac1{\log b} \log \left( b^{(2j-1) L} + 2L \frac{b^{2jL}-1}{b^{2L} -1}
\right) + L = \frac{ \log (x_{,j} + L)}{\log b} + L
\, ,
\end{align*}
for all $j=1, \ldots, n$, where $L$ is the unique solution of
{\rm (\ref{eq3105})}; in particular, $\#{\rm supp}\, \delta_{\bullet}^{\bullet,n}=n$. Moreover, $d_{\sf L}\lt(\B_b,
\delta_{\bullet}^{\bullet, n}\rt)=\omega_bL$, and
$$
\lim\nolimits_{n\to \infty} n d_{\sf L} \lt(\B_b, \delta_{\bullet}^{\bullet,
  n}\rt)  = \frac{\max\{b,2\}-1}{2b-2}\cdot \frac{\log (1 + b \log b)
  -\log (1+\log b)}{\log b} \, .
$$
\end{cor}

To compare  this to Corollary \ref{coLU}, note that $P_{,j}\not
\equiv j/n$ whenever $n\ge 2$,
and then the $n$-th quantization error $d_{\sf L} \lt(\B_b,
\delta_{\bullet}^{\bullet, n}\rt)$ is smaller than the $n$-th 
{\em uniform\/} quantization error $d_{\sf L} \lt(\B_b,
\delta_{\bullet}^{u_n}\rt)$. The $d_{\sf
  L}$-quantization coefficient of $\B_b$ also is smaller than its uniform
counterpart, since
$$
\frac{\log (1 + b \log b)  -\log (1+\log b)}{\log b}  < \frac{b\log b}{1 + b\log
  b} \quad \forall\ b > 1 \, .
$$

\begin{figure}[ht]\label{fig1}
\psfrag{tleg}[l]{\small $b=10, n=3$}
\psfrag{tdu}[l]{\small $d_{\sf L}(\beta_{10},
  \delta_{\bullet}^{u_3})= {\sf L}_{\bullet}(u_3)=1.566\cdot 10^{-1}$}
\psfrag{td}[l]{\small $d_{\sf L}(\beta_{10},
  \delta_{\bullet}^{\bullet,3})={\sf L}_{\bullet}^{\bullet, 3}
  =1.439\cdot 10^{-1}$}
\psfrag{th1}[]{\small $1$}
\psfrag{thx}[]{\small $x$}
\psfrag{tff}[]{\small $F_{\beta_{10}}(x)$}
\psfrag{th10}[]{\small $10$}
\psfrag{tv0}[]{\small $0$}
\psfrag{tv13}[]{\small $\frac13$}
\psfrag{tv23}[]{\small $\frac23$}
\psfrag{tv1}[]{\small $1$}
\begin{center}
\includegraphics[height=7.0cm]{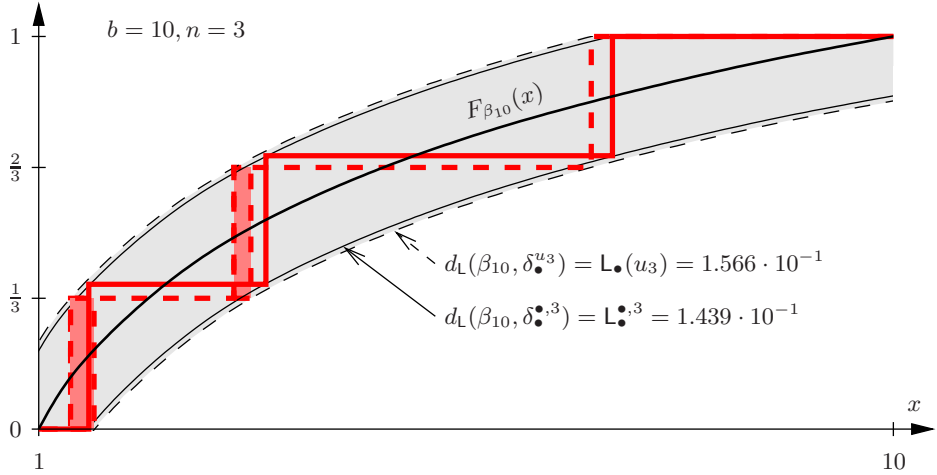}
\end{center}
\vspace*{-6mm}
\caption{The best $d_{\sf L}$-approximation (solid red line) of
  $\beta_{10}$ is unique, whereas best uniform $d_{\sf
  L}$-approximations (broken red lines) are not; see Corollaries
\ref{coL} and \ref{coLU}, respectively.}
\end{figure}

\begin{example}\label{ex310a}
For $\mu = \mbox{\tt Beta}(2,1)$, Theorem \ref{thm3h} yields a unique
best $d_{\sf L}$-approxi\-ma\-ti\-on. Although the equation determining
${\sf L}_{\bullet}^{\bullet, n}$ is less transparent than
(\ref{eq3105}), it can be shown that
$\lim_{n\to \infty} n d_{\sf L} (\mu, \delta_{\bullet}^{\bullet, n}) =
\frac14 (2 - \log 3)< \frac14$.
\end{example}

\begin{example}\label{ex310b}
For the inverse Cantor distribution, a best
$d_{\sf L}$-appro\-xi\-ma\-ti\-on exists by Theorem \ref{thm3h}, and utilizing
the self-similarity of $F_{\mu}^{-1}$, it is possible to derive
estimates such as
\begin{equation}\label{estil}
\frac{1}{216}\le  n^{\log3/ \log 2}d_{\sf
  L}\lt(\mu,\delta_{\bullet}^{\bullet,n}\rt) \le 3 \quad \forall n \in
\N \, , 
\end{equation}
which shows that $\bigl( d_{\sf L} (\mu, \delta_{\bullet}^{u_n})\bigr)$ decays
like $(n^{-\log 3/ \log 2})$, and hence faster than in the case of $\B_b$ and
$\mbox{\tt Beta}(2,1)$.
\end{example}

\section{Kantorovich approximations}\label{secr}

This section studies best finitely supported $d_r$-approximations of
Benford's Law. Mostly, the results are special cases of more
general facts taken from the authors' comprehensive study on
$d_r$-approximations \cite{XB}.

\subsection{$d_1$-approximations}

With $d_{\sf L}$ replaced by $d_1$, the main results of the previous
section have the following analogues, stated here for the reader's
convenience; see \cite[Sec.5]{XB} for details.

\begin{prp}\label{prop4a}
Let $\mu \in \cP$ and $n\in \N$.
\noindent\begin{enumerate}
\item[{\rm(i)}] For every $x\in \Xi_n$, there exists a best $d_1$-approximation
  of $\mu$, given $x$. Moreover, $d_1\lt(\mu, \delta_x^p\rt)= d_1\lt(\mu,
  \delta_x^{\bullet}\rt)$ if and only if, for every $j=0,\ldots, n$,
\begin{equation}\label{eq4101}
x_{,j} < x_{,j+1} \enspace \Longrightarrow \enspace
F_{\mu-} \lt( {\textstyle \frac12} (x_{,j} + x_{,j+1} )\rt) \le
P_{,j} \le  F_{\mu} \lt( {\textstyle \frac12} (x_{,j} +
x_{,j+1})\rt)  \, .
\end{equation}
\item[{\rm (ii)}] For every $p\in \Pi_n$, there exists a best $d_1$-approximation
  of $\mu$, given $p$. Moreover, $d_1\lt(\mu, \delta_x^p\rt)= d_1\lt(\mu,
  \delta_{\bullet}^p\rt)$ if and only if, for every $j=1,\ldots, n$,
\begin{equation}\label{eq4102}
P_{,j-1} < P_{,j} \enspace  \Longrightarrow \enspace
F^{-1}_{\mu-} \lt( {\textstyle \frac12} (P_{,j-1} + P_{,j})\rt)
\le x_{,j} \le  F^{-1}_{\mu} \lt( {\textstyle \frac12} (P_{,j-1} +
P_{,j} )\rt) \, .
\end{equation}
\item[{\rm (iii)}] There exists a best $d_1$-approximation of $\mu$, and if
  $d_{1}\lt(\mu, \delta_x^p\rt) = d_1\lt(\mu, \delta_{\bullet}^{\bullet,n}\rt)$
  then {\rm (\ref{eq4101})} and {\rm (\ref{eq4102})} are valid for every $j=1,\ldots , n$.
\end{enumerate}
\end{prp}
\rb
Though the phrasing of Proposition~\ref{prop4a} emphasizes its analogy
to Theorem~\ref{thm3e} (and also to Theorem~\ref{thloc-K} below),
there nevertheless is a subtle difference: While in \eqref{eq3103}
and \eqref{lK} it can equivalently be stipulated that, respectively,
$\ell_{F_{\mu}, [x_{,j}, x_{,j+1}]} (P_{,j}) \le {\sf L}^{\bullet}
(x)$ and
$F_{\mu-}\lt(x_{,j+1}\rt)-{\sf K}^{\bullet}(x)\le P_{,j}\le
F_{\mu}\lt(x_{,j}\rt)+{\sf K}^{\bullet}(x)$ for {\em all}
$j=0,\ldots,n$, simple examples show that the ``only if'' part of
Proposition~\ref{prop4a}(i) may fail, should \eqref{eq4101} be
replaced by
$$
F_{\mu-} \lt( {\textstyle \frac12} (x_{,j} + x_{,j+1} )\rt) \le
P_{,j} \le  F_{\mu} \lt( {\textstyle \frac12} (x_{,j} +
x_{,j+1})\rt)\quad  \forall\ j=0,\ldots,n.
$$
Similar observations pertain to Proposition~\ref{prop4a}(ii) vis-\`{a}-vis Proposition~\ref{prop3f} and Theorem~\ref{thwei-K}.
\re
Proposition \ref{prop4a} immediately yields the existence of
unique best uniform $d_1$-approximations of $\B_b$; see also \cite[Cor.2.10]{BHM}.

\begin{cor}\label{cor4b}
Let $b>1$ and $n\in \N$. Then the best uniform $d_1$-approximation of\,
$\B_b$ is $\delta_x^{u_n}$, with $x_{,j}= b^{(2j-1)/(2n)}$ for all
$j=1,\ldots, n$, and $\# \mbox{\rm supp}\,
\delta_{\bullet}^{u_n}=n$. Moreover, $d_1(\B_b , \delta_{\bullet}^{u_n}) =
\displaystyle{\frac{1}{\log b} \tanh \left( \frac{\log
      b}{4n}\right)}$, and $\lim_{n\to \infty} n d_1\lt(\B_b,
\delta_{\bullet}^{u_n}\rt) = \frac14$.
\end{cor}

\begin{proof}
By Proposition \ref{prop4a}(ii), $x_{,j}= b^{(2j-1)/(2n)}$ for all
$j=1,\ldots, n$, and
\begin{align*}
nd_1\lt(\B_b, \delta_{\bullet}^{u_n}\rt) & = \frac{n}{b-1}
\sum\nolimits_{j=1}^n \int_{(j-1)/n}^{j/n} \left| b^y -
  b^{(2j-1)/(2n)}\right| \, {\rm d}y \\[2mm] & = \frac{n \lt( b^{1/(4n)} -
  b^{-1/(4n)}\rt)^2}{(b-1)\log b} \sum\nolimits_{j=1}^n
b^{(2j-1)/(2n)}  = \frac{n}{\log b} \tanh \left( \frac{\log b}{4n}\right)
\stackrel{n\to \infty}{\longrightarrow} \frac14 \, .\qquad  
\end{align*}
\end{proof}

\noindent
Best (unconstrained) $d_1$-approximations of $\B_b$ exist and
are unique, too, by virtue of Proposition \ref{prop4a} and a direct calculation.

\begin{cor}\label{cor4c}
Let $b>1$ and $n\in \N$. Then the best $d_1$-approximation of\, $\B_b$
is $\delta_{x}^p$, with
\begin{align*}
x_{,j}  & =  \left( 1 + \frac{j-1}{n} \lt(b^{1/2} - 1\rt)\right) \left( 1 +
  \frac{j}{n} \lt(b^{1/2} - 1\rt)\right) \, , \\ %
\quad P_{,j}  & =
\frac{2}{\log b} \log \left( 1+ \frac{j}{n} \lt(b^{1/2} -1 \rt)\right) \, ,
\end{align*}
for all $j=1,\ldots , n$; in particular, $\# \mbox{\rm supp}\,
\delta_{\bullet}^{\bullet,n}=n$. Moreover, $d_1(\B_b,
\delta_{\bullet}^{\bullet, n}) = \displaystyle{ \frac1{n\log b} \tanh
  \left( \frac{\log b}{4}\right)}$.
\end{cor}

\begin{proof}
Let $\delta_x^p$ be a best
$d_1$-approximation. Then, by Proposition \ref{prop4a}(iii),
$$
b^{P_{,j}} = \frac{x_{,j} + x_{,j+1}}{2} \quad \forall j = 1,
\ldots , n-1 \, ,
$$
but also $x_{,j} = b^{(P_{,j-1} + P_{,j})/2}$ for all $ j=1, \ldots
, n$, and hence $2b^{P_{,j}/2} = b^{P_{,j-1}/2} + b^{P_{j+1}/2}$. Since $P_0
= 0$, $P_n = 1$, it follows that $b^{P_{,j}/2} = 1 + j (b^{1/2}
-1)n^{-1}$ for all $j=0,\ldots , n$. This yields the asserted unique $\delta_x^p$, and
\begin{align*}
d_1\lt(\B_b , \delta_{\bullet}^{\bullet, n}\rt) & = \frac1{b-1} \sum\nolimits_{j=1}^n
\int_{P_{,j-1}}^{P_{,j}} |b^y - x_{,j}| \, {\rm d}y = \frac{b- x_{,n} -
  (x_{,1} -1)}{(b-1)\log b} \\
& = \frac1{n\log b} \tanh
\left( \frac{\log b}{4} \right) \, ,
\end{align*}
via a straightforward calculation. 
\end{proof}

\begin{figure}[ht]\label{fig2}
\psfrag{tleg}[l]{\small $b=10, n=3$}
\psfrag{tdu}[l]{\small $d_1(\beta_{10},
  \delta_{\bullet}^{u_3})= 8.232\cdot 10^{-2}$}
\psfrag{td}[l]{\small $d_1(\beta_{10},
  \delta_{\bullet}^{\bullet,3})=7.520\cdot 10^{-2}$}
\psfrag{th1}[]{\small $1$}
\psfrag{thx}[]{\small $x$}
\psfrag{tff}[]{\small $F_{\beta_{10}}(x)$}
\psfrag{th10}[]{\small $10$}
\psfrag{tv0}[]{\small $0$}
\psfrag{tv13}[]{\small $\frac13$}
\psfrag{tv23}[]{\small $\frac23$}
\psfrag{tv1}[]{\small $1$}
\begin{center}
\includegraphics[height=6.4cm]{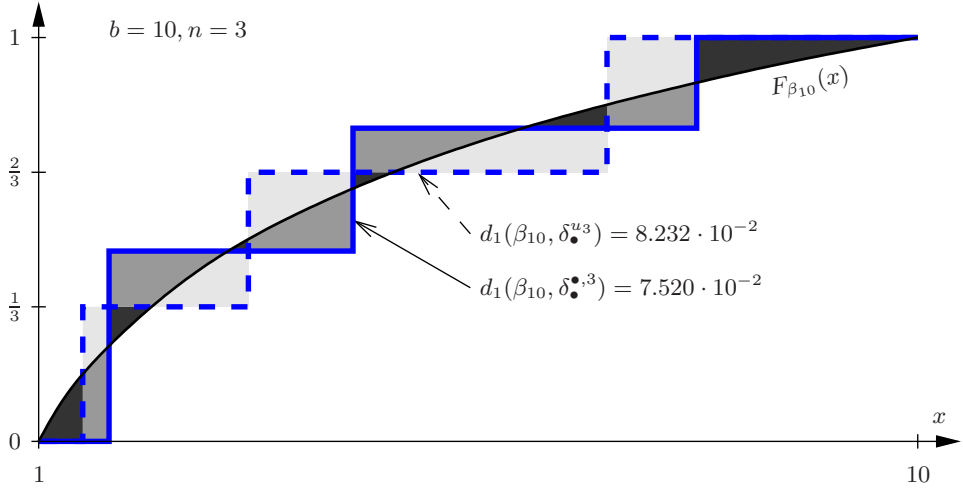}
\end{center}
\vspace*{-4mm}
\caption{The best (solid blue line) and best uniform (broken blue line) $d_1$-approximations of
  $\beta_{10}$ both are unique; see Corollaries \ref{cor4c} and
  \ref{cor4b}, respectively. Coincidentally, best uniform
  $d_1$-approximations of $\beta_{10}$ are best $d_{\sf
    K}$-approximations as well; see Corollary \ref{coKS}.}
\end{figure}

\begin{rem}\label{rm4c}
(i) Due to the highly non-linear nature of the optimality conditions
(\ref{eq4101}) and (\ref{eq4102}), best $d_1$-approximations are rarely given by
explicit formulae such as those in Corollary \ref{cor4c}. Aside from
Benford's Law, the authors know of only two other
families of continuous distributions that allow for similarly explicit formulae,
namely uniform and (one- or two-sided) exponential distributions.

(ii) A popular family of metrics on $\cP$ closely related to $d_1$ are
the so-called {\em Fortet--Mourier $r$-distances\/} ($1\le
r<+\infty$), given by
$$
d_{{\sf FM}_{r}} (\mu, \nu) = \int_{\I} \max\{ 1,|y|\}^{r-1} |F_{\mu}(y) -
F_{\nu}(y)|\, {\rm d}y \, .
$$
Like the L\'evy and Kantorovich metrics, the Fortet--Mourier $r$-distance also
metrizes the weak topology on $\cP$.  The reader is referred to
\cite{PP,R} for details on the mathematical background of $d_{{\sf
    FM}_r}$ and its use for stochastic optimization.  Note that if $\I\subset [1,+\infty[$ then
$$
d_{{\sf FM}_r}(\mu, \nu) = \frac{\lambda \bigl( T(\I)\bigr)}{r} d_1\lt(\mu \circ T^{-1},
\nu \circ T^{-1}\rt) \, ,
$$
with the homeomorphism $T: x\mapsto x^r$ of $[1,+\infty[$. For
instance, $\B_b \circ T^{-1} = \B_{rb}$, and hence best (or best uniform)
$d_{{\sf FM}_{r}}$-approximations of $\B_b$ can easily be
identified using Corollary \ref{cor4c} (or \ref{cor4b}).
\end{rem}

\subsection{$d_r$-approximations ($1<r<+\infty$)}

Similarly to the case of $r=1$, \cite[Thm.5.5]{XB} guarantees that, given any
$n\in \N$, there exists a (unique) best uniform $d_r$-approximation
$\delta_{\bullet}^{u_n}$ of $\B_b$. Except for $r=2$, however, no
explicit formula seems to be available for
$\delta_{\bullet}^{u_n}$. It is desirable, therefore, to at least identify {\em asymptotically\/} best
uniform $d_r$-approximations, that is, a sequence
$(x_n)$ with $x_n \in \Xi_n$ for all $n\in \N$ such that
$$
\lim\nolimits_{n\to \infty} \frac{d_r \lt(\B_b,
  \delta_{x_n}^{u_n}\rt)}{d_r \lt(\B_b, \delta_{\bullet}^{u_n}\rt)} = 1 \, .
$$
Usage of \cite[Thm.5.15]{XB} accomplishes this and also yields
the uniform $d_r$-quanti\-za\-tion coefficient of $\B_b$. (Notice that, as
$r\downarrow 1$, the latter is consistent with Corollary
\ref{cor4b}.)

\begin{prp}\label{prop4d}
Let $b,r>1$. Then $\lt(\delta_{x_n}^{u_n}\rt)$, with $x_{n,j} =
b^{(2j-1)/(2n)}$ for all $n\in \N$ and $j=1, \ldots , n$, is a sequence
of asymptotically best uniform $d_r$-approximations of\,
$\B_b$. Moreover,
$$
\lim\nolimits_{n\to \infty} n d_r (\B_b , \delta_{\bullet}^{u_n}) =
\frac{(\log b)^{1-1/r}}{2(b-1)} \left( \frac{b^r - 1}{r(r+1)}\right)^{1/r} \, .
$$
\end{prp}

The remainder of this section studies best $d_r$-approximations of
$\B_b$. In general, the question of uniqueness of best
$d_r$-approximations is a difficult one, for which only partial
answers exist; see, e.g., \cite[Sec.5]{GL}. Specifically, $\B_b$ does
not seem to satisfy any known condition (such as, e.g.,
log-concavity) that would guarantee uniqueness. However, uniqueness can be
established via a direct calculation.

\begin{theorem}\label{thm4e}
Let $b,r>1$ and $n\in \N$. There exists a unique best
$d_r$-appro\-xi\-ma\-tion $\delta_{\bullet}^{\bullet,n}$ of\, $\B_b$, and $\#{\rm supp}\, \delta_{\bullet}^{\bullet,n}=n$.
\end{theorem}

\begin{proof}
Existence follows as in Theorem \ref{thm3h}; alternatively, see
\cite[Sec.4.1]{GL} or \cite[Prop.5.22]{XB}. To avoid trivialities,
henceforth assume $n\ge 2$. If $d_r\lt(\B_b, \delta_x^p\rt) = d_r \lt(\B_b,
\delta_{\bullet}^{\bullet, n}\rt)$, then by \cite[Thm.5.23]{XB},
$$
b^{P_{,j}} = \frac{x_{,j} + x_{,j+1}}{2} \quad \forall j = 1, \ldots ,
n-1 \, ,
$$
but also
\begin{equation}\label{eq4p01}
\int_{P_{,j-1}}^{\log_b x_{,j}} \lt(x_{,j} - b^y\rt)^{r-1} \, {\rm d}y =
\int_{\log_b x_{,j}}^{P_{,j}} \lt(b^y - x_{,j}\rt)^{r-1} \, {\rm d}y \quad
\forall j=1, \ldots, n \, .
\end{equation}
Eliminating $P$ and substituting $z = b^y/x_{,j}$ in (\ref{eq4p01})
yields $n$ equations for $x_{,1}, \ldots , x_{,n}$, namely
\begin{eqnarray}\label{eq4p02}
\int_{1}^{x_{,1}} (z-1)^{r-1} \frac{{\rm d}z}{z^r} & = & 2^{1-r} g_0
\left( \frac{x_{,2}}{x_{,1}}\right) \, , \nonumber \\
g_r\left( \frac{x_{,j}}{x_{,j-1}}\right) & = & g_0 \left(
  \frac{x_{,j+1}}{x_{,j}}\right) \, , \quad \forall j=2, \ldots , n-1
\, ,\\
g_r\left( \frac{x_{,n}}{x_{,n-1}}\right) & = & g_0 \left(
  \frac{2b - x_{,n}}{x_{,n}}\right) \, ,\nonumber
\end{eqnarray}
where the smooth, increasing function $g_a$, with $a\in \R$, is given
by
$$
g_a(x) = \int_1^x \frac{(z-1)^{r-1}}{z^a (z+1)} \, {\rm d}z \, , \quad
x\ge 1\, .
$$
Assume that $\widetilde{x}\in \Xi_n$ also solves
(\ref{eq4p02}). If $\widetilde{x}_{,1}> x_{,1}$ then
$\widetilde{x}_{,j+1}/\widetilde{x}_{,j} > x_{,j+1}/x_{,j}$ and hence
$\widetilde{x}_{,j+1} > x_{,j+1}$ for all $j=0,\ldots, n-1$, but by
the last equation in (\ref{eq4p02}) also $2b/\widetilde{x}_{,n} > 2b
/x_{,n}$, an obvious contradiction. Similarly, $\widetilde{x}_{,1}<
x_{,1}$ leads to a contradiction. Thus, $\widetilde{x}_{,1} = x_{,1}$,
and consequently $\widetilde{x} = x$. (If $n=1$ then (\ref{eq4p02})
reduces to
$$
\int_{1}^{x_{,1}} (z-1)^{r-1} \frac{{\rm d}z}{z^r} = 2^{1-r} g_0
\left( \frac{2b - x_{,1}}{x_{,1}}\right) \, ,
$$
which also has a unique solution since, as $x_{,1}$ increases from $1$
to $b$, the left side increases from $0$ whereas the right side decreases to $0$.)
In summary, therefore, $x\in \Xi_n$ and $p\in \Pi_n$ are uniquely determined by $d_r\lt(\B_b, \delta_x^p\rt) = d_r \lt(\B_b,
\delta_{\bullet}^{\bullet, n}\rt)$. 
\end{proof}

As in the case of best uniform $d_r$-approximations of $\B_b$, no
explicit formula is available for $\delta_{\bullet}^{\bullet,n}$, not
even when $r=2$. Still, it is possible to identify {\em
  asymptotically\/} best $d_r$-approximations, that is, a
sequence $\lt(\delta_{x_n}^{p_n}\rt)$ with $x_n \in \Xi_n$ and $p_n \in \Pi_n$
for all $n\in \N$ such that
$$
\lim\nolimits_{n\to \infty} \frac{d_r \lt(\B_b, \delta_{x_n}^{p_n}\rt)}{d_r
  \lt(\B_b , \delta_{\bullet}^{\bullet, n}\rt)} = 1 \, .
$$
In addition, the $d_r$-quantization coefficient of $\B_b$ can be
computed explicitly; for details see \cite[Prop.5.26]{XB} and the
references given there. Notice that, as
$r\downarrow 1$, the result is consistent with Corollary \ref{cor4c}.

\begin{prp}\label{prop4f}
Let $b,r>1$. Then $\lt(\delta_{x_n}^{p_n}\rt)$, with
$$
x_{n,j}  = \lt(1+\frac{j}{n+1}\lt(b^{r/(r+1)}-1\rt)\rt)^{1 + 1/r}\, , \quad
P_{n,j}  = \frac1{\log b}  \log \frac{x_{n,j} + x_{n,j+1}}{2}\, ,
$$
for all $n\in \N$ and $j=1, \ldots , n-1$, and $x_{n,n} = \lt(  1 +
(b^{r/(r+1)} -1) \frac{n}{n+1} \rt)^{1+1/r}$, is a sequence of asymptotically
best $d_r$-approximations of\, $\B_b$. Moreover,
$$
\lim\nolimits_{n\to \infty} n d_r (\B_b , \delta_{\bullet}^{\bullet,
  n}) = \frac{r+1}{2(b-1)(\log b)^{1/r}} \left( \frac{b^{r/(r+1)}
  -1 }{r} \right)^{1+1/r}  \, .
$$
\end{prp}

\begin{example}\label{ex46a}
For $\mu = \mbox{\tt Beta}(2,1)$, given any $n\in \N$, a unique best
uniform $d_r$-approximation exists for each $r\ge 1$. The best uniform
$d_1$-approximations $\delta_{x}^{u_n}$, where $x_{,j} =
\sqrt{\frac{2j-1}{2n}}$ for $j=1, \ldots, n$, also constitute a
sequence of asymptotically best uniform $d_r$-approximations for $1<r<2$, with
\begin{equation}\label{eq44add2}
\lim\nolimits_{n\to \infty} n d_r (\mu, \delta_{\bullet}^{u_n}) =
\left( \frac{2^{1-2r}}{(r+1)(2-r)}\right)^{1/r} \, ,
\end{equation}
in analogy to Proposition \ref{prop4d}. For $r\ge 2$, however, this
analogy breaks down, as
$$
\lim\nolimits_{n\to \infty} \frac{n}{\sqrt{\log n}} d_2 (\mu ,
\delta_{\bullet}^{u_n}) = \frac1{4 \sqrt{3}} \, ,
$$
and $\lim_{n\to \infty} n^{1/2 + 1/r} d_r(\mu,
\delta_{\bullet}^{u_n})$ is finite and positive whenever $r>2$.

Since $\mu$ is log-concave, or by an argument similar to the one
proving Theorem \ref{thm4e}, there exists a unique best
$d_r$-approximation of $\mu$. While the authors do not know of an
explicit formula for $\delta_{\bullet}^{\bullet, n}$, simple
asymptotically best $d_r$-approximations in the spirit of Proposition
\ref{prop4f} exist, and
\begin{equation}\label{eq44add3}
\lim\nolimits_{n\to\infty} n d_r (\mu, \delta_{\bullet}^{\bullet,n}) =
2^{1/r -1} \frac{r+1}{(r+2)^{1+1/r}} \quad \forall r \ge 1 \, ;
\end{equation}
see \cite[Ex.5.28]{XB}. Note that (\ref{eq44add3}) is smaller than
(\ref{eq44add2}) for every $1\le r < 2$.
\end{example}

\begin{example}\label{ex47b}
For the inverse Cantor distribution, for every $r\ge 1$ let $\alpha_r = r^{-1} +
(1-r^{-1})\log 2/\log 3$, and note that $\log 2/\log 3
< \alpha_r \le 1$. With this, $3^{\alpha_r} d_r (\mu ,
\delta_{\bullet}^{u_{3n}}) = d_r (\mu, \delta_{\bullet}^{u_n})$ for
all $n\in \N$, and it is readily deduced that
$$
2^{2/r - 4} 3^{-3/r} \le n^{\alpha_r}
d_r (\mu, \delta_{\bullet}^{u_n})  \le 2^{1/r} \quad \forall n \in \N \, .
$$
Thus $\bigl( n^{\alpha_r} d_r (\mu, \delta_{\bullet}^{u_n})\bigr)$ is
bounded below and above by positive constants. (The authors suspect
that this sequence is divergent for every $r\ge 1$.)

Best $d_r$-approximations also exist, and in a similar spirit it can
be shown that $\bigl( n^{\widetilde{\alpha}_r} d_r (\mu,
\delta_{\bullet}^{\bullet, n}) \bigr)$ is bounded below and above by
positive constants (and again, presumably, divergent), where $\widetilde{\alpha}_r = \alpha_r \log 3 /
\log 2$. Note that $1< \widetilde{\alpha}_r \le \log 3/\log 2$, and
hence $\bigl( d_r (\mu, \delta_{\bullet}^{\bullet,n})\bigr)$ decays
faster than $(n^{-1})$ for every $r\ge 1$.
\end{example}

\section{Kolmogorov approximations}\label{secK}

This section discusses best finitely supported $d_{\sf
  K}$-approximations. Though ultimately the results are true analogues
of their counterparts in Sections~3 and 4, the underlying arguments
are subtly different, which may be seen as a reflection of the fact
that $d_{\sf K}$ metrizes a topology finer than the weak topology of
$\cP$. (Recall, however, that $d_{\sf K}$ does metrize the weak topology on $\cP_{\sf cts}$.)

Given $\mu\in \cP$ and $n\in\N,$ for every $x\in\Xi_n,$
let
$$
{\sf K}^{\bullet}(x)=\max\lt\{F_{\mu-}(x_{,1}),\frac{1}{2}{\max}
_{j=1}^{n-1}\ \bigl( F_{\mu-}(x_{,j+1})-F_{\mu}(x_{,j})\bigr) ,
1-F_{\mu}(x_{,n}) \rt\}.
$$
Note that ${\sf K}^{\bullet}(x)=d_{\sf  K}\lt(\mu,\delta_x^{\pi(x)}\rt)$ with
$\Pi(x)_{,j}=\frac{1}{2}\bigl(F_{\mu}(x_{,j})+F_{\mu-}(x_{,j+1})\bigr)$
for all $j=1,\ldots,n-1$. Existence and characterization of
best $d_{\sf K}$-approximations with prescribed locations are analogous to Theorem~\ref{thm3e}.
\begin{theorem}\label{thloc-K}
Assume that $\mu\in\cP,$ and $n\in\N.$ For every $x\in \Xi_n,$ there
exists a best $d_{\sf K}$-approximation of $\mu,$ given $x.$ Moreover,
$d_{\sf K}\lt(\mu,\delta_x^p\rt)=d_{\sf K}\lt(\mu,\delta_x^{\bullet}\rt)$ if and only if, for every $j=0,\ldots, n$,
\eqb\label{lK}
x_{,j} < x_{,j+1} \enspace \Longrightarrow \enspace
F_{\mu-}\lt(x_{,j+1}\rt)-{\sf K}^{\bullet}(x)\le P_{,j}\le F_{\mu}\lt(x_{,j}\rt)+{\sf K}^{\bullet}(x),\, \eqe and in this case $d_{\sf K}\lt(\mu,\delta_x^{\bullet}\rt)={\sf K}^{\bullet}(x)$.
\end{theorem}

\begin{proof}
Given $x\in\Xi_n$ and $p\in\Pi_n$, let $y\in\Xi_m$ and $q\in\Pi_{m}$ as in the proof of Lemma~\ref{lem3d}. Then
\[\begin{split}
  d_{\sf K}& \lt(\mu,\delta_x^p\rt) =\max\nolimits_{i=0}^m\sup\nolimits_{t\in[y_{,i},y_{
  ,i+1}[}\lt|F_{\mu}(t)-Q_{,i}\rt|\\
& \ge\max\lt\{F_{\mu-}(y_{,1}),\frac{1}{2}\max\nolimits_{i=1}^{m-1}\lt(F_{\mu-}
(y_{,i+1})-F_{\mu}(y_{,i})\rt),1-F_{\mu}(y_{,m})\rt\}\\
& = \max\lt\{F_{\mu-}(x_{,1}),\frac{1}{2}\max\nolimits_{j=1}^{n-1}\lt(F_{\mu-}
(x_{,j+1})-F_{\mu}(x_{,j})\rt),1-F_{\mu}(x_{,n})\rt\}\\
&  =  {\sf K}^{\bullet}(x).\end{split}\]
This shows that $\delta_x^{\pi(x)}$ is a best $d_{\sf K}$-approximation, given $x$, and $d_{\sf K}\lt(\mu,\delta_x^{\bullet}\rt)={\sf K}^{\bullet}(x)$. Moreover, $d_{\sf K}\lt(\mu,\delta_x^p\rt)={\sf K}^{\bullet}(x)$ if and only if
\[\max\lt\{\lt|F_{\mu-}(y_{,i+1})-Q_{,i}\rt|,\lt|F_{\mu}
  (y_{,i})-Q_{,i}\rt|\rt\}\le {\sf K}^{\bullet}(x)\,\quad \forall\
  i=1,\ldots,m-1,\] 
that is, \[F_{\mu-}(y_{,i+1})-{\sf K}^{\bullet}(x)\le Q_{,i}\le F_{\mu}(y_{,i})+{\sf K}^{\bullet}(x)\quad \forall\ i=0,\ldots,m,\] which in turn is equivalent to the validity
\eqref{lK} for every $j$. 
\end{proof}

To address the approximation problem with prescribed weights, an
auxiliary function analogous to $\ell_{f,I}$ in Section 3 is useful.
Specifically, given a non-decreasing function $f:\R\to\overline{\R}$,
let $I\subset\R$ be any bounded, non-empty interval, and define
$\kappa_{f,I}:\R\to\overline{\R}$ as\[\kappa_{f,I}(x)=\max\lt\{\bigl|f_-(x)-\inf I\bigr|,\bigl|f_+(x)-\sup I\bigr|\rt\}.\]
A few basic properties of $\kappa_{f,I}$ are easily established.
\begin{prp}\label{le1}
Let $f:\R\to\overline{\R}$ be non-decreasing, and $\varnothing\neq
I\subset\R$ a bounded interval. Then, with
$s:=f^{-1}\lt(\frac{1}{2}(\inf I+\sup I)\rt)$, the function
$\kappa_{f,I}$ is non-increasing on $]$$-\infty,s[,$ and
non-decreasing on $]s,+\infty[$. Moreover, $\kappa_{f,I}$ attains a
minimal value whenever $\inf I\le\frac{1}{2}\bigl(f_-(s)+f_+(s)\bigr)\le\sup I$.
\end{prp}

It is worth noting that $\kappa_{f,I}$ may in general not attain its infimum, as the example of $f=15F_{\mu}$, with $\mu=\frac{1}{15}\lambda\lt|_{[0,5]}\rt.+\frac{2}{3}\delta_5$, and $I=[6,8]$ shows, for which $s=5$, and $\kappa_{f,I}(5-)=3$, $\kappa_{f,I}(5)=7$, $\kappa_{f,I}(5+)=9$; correspondingly, $\frac{1}{2}\bigl(f_-(5)+f_+(5)\bigr)\notin I$.

By using functions of the form $\kappa_{f,I}$, the value of $d_{\sf K}(\mu,\nu)$ can easily be bounded above whenever $\nu$ has finite support. For convenience, for every $n\in\N$ let $\Xi_n^+=\lt\{x\in\Xi_n:\ x_{,1}<\ldots<x_{,n}\rt\}$. The proof of the following analogue of Lemma~\ref{lem3d} is straightforward.

\begin{prp}\label{prop5a}
Let $\mu\in\cP$ and $n\in\N$. For every $x\in\Xi_n$ and $p\in\Pi_n$,
\eqb\label{K-in}
d_{\sf K}\bigl( \mu, \delta_x^p \bigr)\le
\max\nolimits_{j=1}^n\kappa_{F_{\mu},[P_{,j-1},P_{,j}]}(x_{,j}),
\eqe 
and equality holds in \eqref{K-in} whenever $x\in\Xi_n^+$.
\end{prp}

Consider for instance $\mu=\frac{1}{6}\lambda\lt|_{[0,2]}\rt.+\frac{2}{3}\delta_1$, and $x=(1,1)$. Then, for every $p\in\Pi_2$, clearly $d_{\sf K}\lt(\mu,\delta_x^p\rt)=\frac{1}{6}$, whereas $\max\nolimits_{j=1}^2\kappa_{F_{\mu},[P_{,j-1},P_{,j}]}(x_{,j})=\frac{1}{3}+
\lt|p_{,1}-\frac{1}{2}\rt|\ge\frac{1}{3}$. Thus the inequality
\eqref{K-in} may be strict if $x\notin\Xi_n^+$. This, together with
the fact that a function $\kappa_{f,I}$ may not attain its infimum,
suggests that $d_{\sf K}$-approximations with prescribed weights are
potentially somewhat fickle. Still, best approximations do exist and can be
characterized in a spirit similar to Sections 3 and 4. To this end,
given $\mu\in\cP$ and $n\in\N$, for every $p\in\Pi_n$, let
$$
{\sf K}_{\bullet}(p)=d_{\sf K}\lt(\mu,\delta_{\xi(p)}^p\rt)\,\
\text{with}\ \
\xi(p)_{,j}=F_{\mu}^{-1}\lt(\frac{1}{2}\lt(P_{,j-1}+P_{,j}\rt)\rt)\quad
\forall\ j=1,\ldots,n.
$$
Note that ${\sf K}_{\bullet}(p)\le\frac{1}{2}\max_{j=1}^np_{,j}$, and
in fact ${\sf K}_{\bullet}(p)=\frac{1}{2}\max_{j=1}^np_{,j}$ whenever
$\mu\in\cP_{\sf cts}$.

\begin{theorem}\label{thwei-K}
Assume that $\mu\in\cP$, and $n\in\N.$ For every $p\in \Pi_n,$ there exists a best $d_{\sf K}$-approximation of $\mu,$ given $p$. Moreover, $d_{\sf K}\lt(\mu,\delta_x^p\rt)=d_{\sf K}\lt(\mu,\delta^p_{\bullet}\rt)$ if and only if, for every $j=1,\ldots, n$,
\eqb\label{wK}
P_{,j-1} < P_{,j} \enspace  \Longrightarrow \enspace
F_{\mu-}^{-1}\bigl( P_{,j}-{\sf K}_{\bullet}(p)\bigr) \le 
x_{,j}\le F_{\mu}^{-1}\bigl( P_{,j-1}+{\sf K}_{\bullet}(p)\bigr),
\eqe 
and in this case $d_{\sf K}\lt(\mu,\delta_{\bullet}^p\rt)={\sf K}_{\bullet}(p).$
\end{theorem}
\begin{proof}
Note first that deleting all zero entries of $p$ does not change the value of ${\sf K}_{\bullet}(p)$, and hence does not affect \eqref{wK}, nor of course the asserted existence of a best $d_{\sf K}$-approximation, given $p$. Thus assume $\min_{j=1}^np_{,j}>0$ throughout. For convenience, write $\xi(p)$ simply as $\xi$, and for every $x\in\Xi_n$, write $F_{\delta_x^p}$ as $G$. To prove the existence of a best $d_{\sf K}$-approximation of $\mu$, given $p$, as well as $d_{\sf K}\lt(\mu,\delta_{\bullet}^p\rt)={\sf K}_{\bullet}(p)$, clearly it suffices to show that \eqb\label{eq5100}
d_{\sf K}\lt(\mu,\delta_x^p\rt)\ge d_{\sf K}\lt(\mu,\delta_{\xi}^p\rt)\quad\, \forall\ x\in\Xi_n.
\eqe 
Similarly to the proof of Lemma~\ref{lem3d}, label $\xi$ uniquely as
\begin{align*}
\xi_{,1} = \ldots = \xi_{,j_1} & < \xi_{,j_1+1} = \ldots = \xi_{,j_2}< \xi_{,j_2 +
1} = \ldots \\
& < \ldots = \xi_{,j_{m-1}}< \xi_{,j_{m-1}+1} = \ldots = \xi_{,j_m}
\, ,
\end{align*}
with integers $i\le j_i \le m$ for $1\le i\le m$, and $j_0 = 0$, $j_m
= n$, and define $\eta \in \Xi_m$ and $ q\in \Pi_m$ as
$\eta_{,i} = \xi_{,j_i}$ and $q_{,i} = P_{,j_i} - P_{,
  j_{i-1}}$, respectively. With this, $\delta_{\xi}^p=\delta_{\eta}^q$, and by Proposition~\ref{prop5a},
\[{\sf K}_{\bullet}(p)=d_{\sf K}\lt(\mu,\delta_{\eta}^q\rt)=\max\nolimits_{i=1}^m\kappa_{F_{\mu},\lt[Q_{,i-1
},Q_{,i}\rt]}\lt(\eta_{,i}\rt).\]
Pick $i$ such that $\kappa_{F_{\mu},\lt[Q_{,i-1}
,Q_{,i}\rt]}\lt(\eta_{,i}\rt)={\sf K}_{\bullet}(p)$, that is,
$$\max\lt\{\lt|F_{\mu
-}\lt(\eta_{,i}\rt)-Q_{,i-1}\rt|,\lt|F_{\mu}\lt(\eta_{,i}\rt)-Q_{,i}\rt|\rt\}=
{\sf K}_{\bullet}(p).$$
Clearly, to establish \eqref{eq5100} it is enough to show that
\eqb\label{eq5101}\max\lt\{\lt|F_{\mu
-}\lt(\eta_{,i}\rt)-G_-\lt(\eta_{,i}\rt)\rt|,\lt|F_{\mu}\lt(\eta_{,i}\rt)-G
\lt(\eta_{,i}\rt)\rt|\rt\}\ge{\sf K}_{\bullet}(p)\eqe and this will now be done. To this end, notice that by the definition of $\eta$,
\eqb\label{eq5102}
\frac{1}{2}\lt(P_{,j_{i-1}-1}+P_{,j_{i-1}}\rt)\le F_{\mu-}\lt(\eta_{,i}\rt)\le\frac{1}{2}\lt(P_{,j_{i-1}}+P_{,
j_{i-1}+1}\rt),\eqe but also
\eqb\label{eq5103}
\frac{1}{2}\lt(P_{,j_i-1}+P_{,j_i}\rt)\le F_{\mu}\lt(\eta_{,i}\rt)\le\frac{1}{2}\lt(P_{,j_i}+P_{,
j_i+1}\rt),\eqe
with the convention that $P_{,-1}=0$ and $P_{,n+1}=1$.

Assume first that ${\sf K}_{\bullet}(p)=\lt|F_{\mu
-}(\eta_{,i})-Q_{,i-1}\rt|$. If $\eta_{,i}\le x_{,j_{i-1}}$ then
$G_-\lt(\eta_{,i}\rt)\le P_{,j_{i-1}-1}$, and hence
$F_{\mu-}(\eta_{,i})-G_-\lt(\eta_{,i}\rt)\ge
F_{\mu-}\lt(\eta_{,i}\rt)-P_{,j_{i-1}}$, but also, by \eqref{eq5102},
\begin{align*}
F_{\mu-}\lt(\eta_{,i}\rt)-G_-\lt(\eta_{,i}\rt) & \ge
F_{\mu-}\lt(\eta_{,i}\rt)-P_{,j_{i-1}} -
\lt(2F_{\mu-}\lt(\eta_{,i}\rt)-P_{,j_{i-1}-1}-P_{,j_{i-1}}\rt)\\
& =P_{,j_{i-1}}-F_{\mu-}\lt(\eta_{,i}\rt) \, ,
\end{align*}
and consequently \[F_{\mu-}\lt(\eta_{,i}\rt)-G_-\lt(\eta_{,i}\rt)\ge \lt|F_{\mu
-}\lt(\eta_{,i}\rt)-P_{,j_{i-1}}\rt|=\lt|F_{\mu
-}\lt(\eta_{,i}\rt)-Q_{,i-1}\rt|={\sf K}_{\bullet}(p).\] If
$x_{,j_{i-1}}<\eta_{,i}\le x_{,j_{i-1}+1}$ then
$G_-\lt(\eta_{,i}\rt)=P_{,j_{i-1}}$ and hence 
$$\lt|F_{\mu-}\lt(\eta_{,i}\rt)-G_-\lt(\eta_{,i}\rt)\rt|={\sf
  K}_{\bullet}(p).
$$
Finally, if $\eta_{,i}>x_{,j_{i-1}+1}$ then $G_-\lt(\eta_{,i}\rt)\ge P_{,j_{i-1}+1}$, and hence $G_-\lt(\eta_{,i}\rt)-F_{\mu-}\lt(\eta_{,i}\rt)\ge P_{,j_{i-1}}-F_{\mu-}\lt(\eta_{,i}\rt)$, but also, again by \eqref{eq5102},
\begin{align*}
G_-\lt(\eta_{,i}\rt)-F_{\mu-}\lt(\eta_{,i}\rt) & \ge P_{,j_{i-1}+1}-F_{\mu-}\lt(\eta_{,i}\rt)-\lt(P_{,j_
{i-1}}+P_{,j_{i-1}+1}-2F_{\mu-}\lt(\eta_{,i}\rt)\rt)\\
& =F_{\mu-}\lt(\eta_{,i}\rt)
-P_{,j_{i-1}},
\end{align*}
and therefore \[G_-\lt(\eta_{,i}\rt)-F_{\mu-}\lt(\eta_{,i}\rt)\ge\lt|F_{\mu-}\lt(\eta_{,i}\rt)
-P_{,j_{i-1}}\rt|={\sf K}_{\bullet}(p).\]
Thus \eqref{eq5101} holds whenever ${\sf K}_{\bullet}(p)=\lt|F_{\mu
-}\lt(\eta_{,i}\rt)-Q_{,i-1}\rt|$.

Next assume that ${\sf
  K}_{\bullet}(p)=\lt|F_{\mu}\lt(\eta_{,i}\rt)-Q_{,i}\rt|$. Utilizing
\eqref{eq5103} instead of \eqref{eq5102}, completely analogous
arguments show that
$\lt|F_{\mu}\lt(\eta_{,i}\rt)-G\lt(\eta_{,i}\rt)\rt|\ge{\sf
  K}_{\bullet}(p)$ in this case as well, which again implies
\eqref{eq5101}. The latter therefore holds in either case.
As seen earlier, this proves the existence of a best $d_{\sf K}$-approximation of $\mu$, given $p$, and also that $d_{\sf K}\lt(\mu,\delta_{\bullet}^p\rt)={\sf K}_{\bullet}(p)$.

Finally, with $y\in\Xi_m^+$ and $p\in\Pi_m$ as in the proof of
Lemma~\ref{lem3d}, observe that $d_{\sf
  K}\lt(\mu,\delta_x^p\rt)= {\sf K}_{\bullet}(p)$ if and only if
$\max_{i=1}^m\kappa_{F_{\mu},\lt[Q_{,i-1},Q_{,i}\rt]}(y_{,i})={\sf
  K}_{\bullet}(p)$, by Proposition~\ref{prop5a}. As seen in the proof
of Theorem~\ref{thloc-K}, this means that
$$
F_{\mu-}(y_{,i+1})-{\sf K}_{\bullet}(p)\le Q_{,i}\le
F_{\mu}(y_{,i})+{\sf K}_{\bullet}(p)\quad \forall\ i=0,\ldots,m\, ,
$$
or equivalently,
$$
F_{\mu-}^{-1}\lt(Q_{,i}-{\sf K}_{\bullet}(p)\rt)\le y_{,i}\le
F_{\mu}^{-1}\lt(Q_{,i-1}+{\sf K}_{\bullet}(p)\rt)\quad \forall\
i=1,\ldots,m\, ,
$$
which in turn is equivalent to the validity of \eqref{wK} for every $j$. 
\end{proof}
\begin{cor}\label{cont}
Assume $\mu\in\cP_{\sf cts}$, and $n\in\N$.
Then $d_{\sf K}\lt(\mu,\delta_x^{u_n}\rt)\ge\frac{1}{2}n^{-1}$ for all
$x\in\Xi_n$, with equality holding if and only if
$$
F_{\mu-}^{-1}\lt(\frac{2j-1}{2n}\rt)\le x_{,j}\le
F_{\mu}^{-1}\lt(\frac{2j-1}{2n}\rt) \quad\forall j=1,\ldots,n.
$$
\end{cor}

By combining Theorems \ref{thloc-K} and \ref{thwei-K}, it is possible to
characterize best $d_{\sf K}$-approximations of $\mu \in \cP$ as
well. For this, associate with every non-decreasing function $f:\R\to\overline{\R}$  and every number $a\ge 0$ a
map $S_{f,a}:\overline{\R} \to \overline{\R}$, given by
$$
S_{f, a} (x) = f_+ \lt(f^{-1}(x+a)\rt) +a \quad
\forall\ x \in \overline{\R} \, .
$$
This map is a true analogue of $T_{f,a}$ in Section 3, and in fact, Proposition~\ref{prop3g}, with $T_{f,a}$ replaced by $S_{f,a}$, remains fully valid. Identical reasoning then shows that
$$
{\sf K}_{\bullet}^{\bullet, n} := \min \lt\{a \ge 0 : S_{F_{\mu},
  a}^{[n]}(0) \ge 1\rt\} < +\infty \, ;
$$
again, $\lt({\sf K}_{\bullet}^{\bullet, n}\rt)$ is
non-increasing, $n{\sf K}_{\bullet}^{\bullet,n}\le \frac12$ for every
$n$, and ${\sf K}_{\bullet}^{\bullet, n}=0$ if and only if $\#
\mbox{\rm supp}\, \mu \le n$. Notice that if $\mu\in\cP_{\sf cts}$ then
$$
S_{F_{\mu}, a}(x) = \left\{
\begin{array}{ll}
a & \mbox{\rm if } x< - a \, , \\
2a + x & \mbox{\rm if } -a \le
x <1 -a \, , \\
a + 1 & \mbox{\rm if } x\ge 1 - a \, ,
\end{array}
\right.
$$
from which it is clear that ${\sf K}_{\bullet}^{\bullet, n}=\frac{1}{2}n^{-1}$.

\begin{theorem}\label{bestK}
Let $\mu \in \cP$ and $n\in \N$. There exists a best $d_{\sf
  K}$-approximation of $\mu$, and $d_{\sf K} \lt(\mu, \delta_{\bullet}^{\bullet, n}\rt) =
{\sf K}_{\bullet}^{\bullet, n}$. Moreover, for every $x\in \Xi_n$ and $p\in \Pi_n$,
the following are equivalent:
\noindent\begin{enumerate}
\item[{\rm(i)}] $d_{\sf K} \lt(\mu, \delta_x^p \rt) = d_{\sf K}
  \lt(\mu, \delta_{\bullet}^{\bullet, n}\rt)$;
\item[{\rm (ii)}] all implications in {\rm (\ref{lK})} are valid with ${\sf
    K}^{\bullet}(x)$ replaced by ${\sf K}^{\bullet, n}_{\bullet}$;
\item[{\rm (iii)}] all implications in {\rm (\ref{wK})} are valid with ${\sf
    K}_{\bullet}(p)$ replaced by ${\sf K}^{\bullet, n}_{\bullet}$.
\end{enumerate}
\end{theorem}
\begin{proof}
Note that once the existence of a best $d_{\sf K}$-approximation
of $\mu$ is established, the proof is virtually identical to that of
Theorem~\ref{thm3h}. Thus, only the existence is to be proved here.
To this end, let $a=\inf\nolimits_{x\in\Xi_n,p\in\Pi_n}d_{\sf
  K}\lt(\mu,\delta_x^p\rt)$, and pick sequences $(x_k)$ and $(p_k)$ in
$\Xi_n$ and $\Pi_n$, respectively, with the property that 
$\lim_{k\to \infty} d_{\sf K}\lt(\mu,\delta_{x_k}^{p_k}\rt)=  a$. By
the compactness of $\Xi_n$, assume w.o.l.g.\ that $\lim_{k\to \infty}
x_k\ = \eta \in\Xi_n$. Since $a\le {\sf K}^{\bullet}(x_k)\le d_{\sf
  K}\lt(\mu,\delta_{x_k}^{p_k}\rt)$, it suffices to show that ${\sf
  K}^{\bullet}(\eta)\le a$. To see the latter, assume that
$\eta_{,j}<\eta_{,j+1}$ for any $j=1,\ldots,n-1$. Then
$x_{k,j}<x_{k,j+1}$ for all sufficiently large $k$, and hence by
Theorem~\ref{thloc-K}, $F_{\mu-}(x_{k,j+1})-F_{\mu}(x_{k,j})\le2{\sf
  K}^{\bullet}(x_k)$, which in turn implies
$$
F_{\mu-}(\eta_{,j+1})-F_{\mu}(\eta_{,j})\le\liminf\nolimits_{k\to
\infty}\lt(F_{\mu-}(x_{k,j+1})-F_{\mu}(x_{k,j})\rt)  \le2 a \, .
$$
Since, similarly, $F_{\mu-}\lt(\eta_{,1}\rt)\le a$ and
$1-F_{\mu}\lt(\eta_{,n}\rt)\le a$, it follows that ${\sf
  K}^{\bullet}(\eta)\le a$, as claimed. 
\end{proof}
\begin{cor}\label{contbest}
Assume $\mu\in\cP_{\sf cts}$, and $n\in\N$. Then
${\sf K}_{\bullet}^{\bullet,n}={\sf K}_{\bullet}(u_n)=\frac{1}{2}n^{-1}$,
and $\delta_x^p$ with $x\in\Xi_n$, $p\in\Pi_n$ is a best $d_{\sf
  K}$-approximation of $\mu$ if and only if it is a best uniform
$d_{\sf K}$-approximation of $\mu$.
\end{cor}

\rb\label{rb1}
{\rm (i)} By Theorem~\ref{bestK}, ${\sf K}_{\bullet}^{\bullet,n}=\min_{x\in\Xi_n}{\sf K}^{\bullet}(x)=\min_{p\in\Pi_n}{\sf K}_{\bullet}(p)$.

{\rm (ii)} If $\mu$ has even a single atom, then ${\sf
  K}_{\bullet}^{\bullet,n}$ may be smaller than ${\sf
  K}_{\bullet}(u_n)$, and thus a best uniform $d_{\sf
  K}$-approximation may not be a best $d_{\sf K}$-approximation.
A simple example illustrating this is
$\mu=\frac{3}{4}\delta_0+\frac{1}{4}\lambda\lt|_{[0,1]}\rt.$, where
${\sf K}_{\bullet}^{\bullet,n}= \frac{1}{4}(2n-1)^{-1}$ whereas ${\sf
  K}_{\bullet}(u_n)=\frac1{2} \max \{n,2\}^{-1}$, and hence ${\sf K}^{\bullet,n}_{\bullet}<{\sf K}_{\bullet}(u_n)$ for every $n\ge2$.
\re

For Benford's Law, the best $d_{\sf K}$-approximations are the same as
the best uniform $d_1$-approxi\-ma\-ti\-ons; see also 
Figure 1.

\begin{cor}\label{coKS}
Assume $b>1,$ and $n\in\N.$ Then $\delta_{x_n}^{u_n}$ with
$x_{n,i}=b^{(2j-1)/(2n)}$ for all $j = 1, \ldots , n$ is the unique best
(uniform) $d_{\sf K}$-approximation of $\beta_b.$ Moreover, $d_{\sf K}\lt(\beta_b,\delta_{\bullet}^{\bullet,n}\rt)=\frac{1}{2}n^{-1}$.
\end{cor}

\begin{example}\label{ex58a}
For $\mu = \mbox{\tt Beta}(2,1)$, both $F_{\mu}$ and $F_{\mu}^{-1}$
are continuous. By Corollaries \ref{cont} and \ref{contbest}, the best
(or best uniform) $d_{\sf K}$-approximation of $\mu$ is
$\delta_{x}^{u_n}$, with $x_{,j} = \sqrt{\frac{2j-1}{2n}}$ for $j=1,
\ldots , n$, and $d_{\sf K} (\mu, \delta_{\bullet}^{u_n}) = d_{\sf K}
(\mu, \delta_{\bullet}^{\bullet, n}) = \frac12 n^{-1}$. With Examples
\ref{ex37a}, \ref{ex310a}, and \ref{ex46a}, therefore, the sequences
$\bigl( n d_* (\mu, \delta_{\bullet}^{\bullet, n})\bigr)$ all converge
to a finite, positive limit, and so do $\bigl( n d_* (\mu,
\delta_{\bullet}^{u_n})\bigr)$, provided that $r<2$ in case $* = r$.
\end{example}

\begin{example}\label{ex58b}
Even though the inverse Cantor distribution is discrete with
infinitely many atoms, a best uniform $d_{\sf K}$-approximation
exists, by Theorem \ref{thwei-K}. Utilizing (\ref{eq3}), a tedious but elementary analysis of
$F_{\mu}$ reveals that (\ref{eq36n1}) is valid with $d_{\sf K}$
instead of $d_{\sf L}$. With Examples \ref{ex37b} and \ref{ex47b},
therefore, $\bigl(n d_*(\mu, \delta_{\bullet}^{u_n})\bigr)$ is bounded below and above
by positive constants for $* = {\sf L}, 1 , {\sf K}$, but tends to
$+\infty$ for $* =r>1$.

Very similarly, a best $d_{\sf K}$-approximation exists, by Theorem
\ref{bestK}, and the estimates (\ref{estil}) hold with $d_{\sf K}$
instead of $d_{\sf L}$. Thus, $\bigl( n^{\log 3 / \log
  2}  d_*(\mu, \delta_{\bullet}^{\bullet, n} ) \bigr)$ is bounded below
and above by positive constants for $* = {\sf L}, 1, {\sf K}$, but
tends to $+\infty$ for $* = r>1$.
\end{example}

\section{Conclusion}\label{secC}

As the title of this article suggests, and the introduction explains,
the general results have been motivated by a quantitative analysis of
Benford's Law, and the precise statements regarding the latter are but
simple corollaries of the former. In particular, Sections \ref{secL}
to \ref{secK} show that the quantization coefficients $Q_* = \lim_{n\to
\infty} n d_* (\beta_b, \delta_{\bullet}^{\bullet, n})$ and their
uniform counterparts $Q_{*,u} = \lim_{n\to
\infty} n d_* (\beta_b, \delta_{\bullet}^{u_n})$ all are finite
and positive for each metric $d_*$ considered. Clearly, $Q_* \le Q_{*,u}$ for all $b>1$. Also, note that $\bigl(
 n d_* (\B_b, \delta_{\bullet}^{\bullet, n} )\bigr)$ is
 non-increasing, possibly constant, whereas
$\bigl( n d_* (\B_b, \delta_{\bullet}^{u_n} )\bigr)$ is
non-decreasing. Figure \ref{fig3} summarizes the results obtained earlier.

\begin{figure}[ht]
\psfrag{t11}[]{\small $*$}
\psfrag{t12}[]{\small $Q_*$}
\psfrag{t13}[]{\small $Q_{*,u}$}
\psfrag{t21}[]{\textcolor{red}{\small ${\sf L}$}}
\psfrag{t22}[]{\textcolor{red}{\small $\displaystyle \frac{\max \{b,2\} - 1}{2b-2} \cdot
 \frac{\log (1 + b \log b) - \log (1 + \log b)}{\log b}$}}
\psfrag{t23}[]{\textcolor{red}{\small $\displaystyle \frac{\max \{b,2\} - 1}{2b-2} \cdot  \frac{b\log b}{1+b\log b}$}}
\psfrag{t31}[]{\textcolor{blue}{\small $r\ge 1$}}
\psfrag{t32}[]{\textcolor{blue}{\small $\displaystyle \frac{r+1}{2(b-1)(\log
      b)^{1/r}} \left( \frac{b^{r/(r+1)} -1 }{r} \right)^{1+1/r}$}}
\psfrag{t33}[]{\textcolor{blue}{\small $\displaystyle \frac{(\log b)^{1-1/r}}{2(b-1)} \left( \frac{b^r - 1}{r(r+1)}\right)^{1/r}$}}
\psfrag{t41}[]{\small ${\sf K}$}
\psfrag{t42}[]{\small $\displaystyle \frac12$}
\psfrag{t43}[]{\small $\displaystyle \frac12$}
\begin{center}
\includegraphics[height=5.0cm]{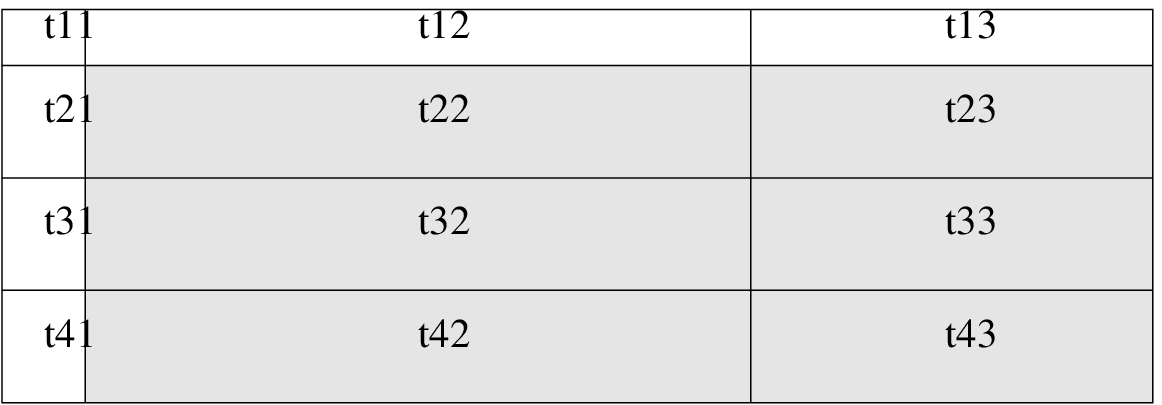}
\end{center}
\caption{The quantization ($Q_*$) and uniform quantization ($Q_{*,u}$) coefficients
  of $\beta_b$ for $d_*$; see also Figure \ref{fig4}.}\label{fig3}
\end{figure}

The dependence of $Q_*$ and $Q_{*,u}$ on $b$ is illustrated in Figure
\ref{fig4}. On the one hand, $Q_{\sf L}$ and $Q_{{\sf L},u}$ tend to
$\frac12$ as $b\downarrow 1$, but also as $b \to +\infty$, both
attaining their respective minimal value for $b=2$. On the other hand,
$Q_r$ and $Q_{r,u}$ both tend to $\frac12 (r+1)^{-1/r}$ as
$b\downarrow 1$, whereas $\lim_{b\to +\infty} (\log b)^{1/r} Q_r =
\frac12 (r+1) r^{-(r+1)/r}$ and $\lim_{b\to +\infty} (\log b)^{1/r -
  1} Q_{r,u} = \frac12 r^{-1/r} (r+1)^{-1/r}$. Finally, $Q_{\sf K} =
Q_{{\sf K},u} = \frac12$ for all $b$.

\begin{figure}[!ht]
\psfrag{tb}[]{\small  $b$}
\psfrag{tq1}[l]{\textcolor{blue}{\small $Q_1$}}
\psfrag{tq2}[l]{\textcolor{cyan}{\small $Q_2$}}
\psfrag{tq1u}[l]{\textcolor{blue}{\small $Q_{1,u}$}}
\psfrag{tq2u}[l]{\textcolor{cyan}{\small $Q_{2,u}$}}
\psfrag{tleg}[l]{\small $Q_{\sf K} = Q_{{\sf K},u}$}
\psfrag{tql}[l]{\textcolor{red}{\small $Q_{\sf L}$}}
\psfrag{tqlu}[l]{\textcolor{red}{\small $Q_{{\sf L},u}$}}
\psfrag{th1}[]{\small $1$}
\psfrag{th10}[]{\small $10$}
\psfrag{th100}[]{\small $100$}
\psfrag{th1000}[]{\small $1000$}
\psfrag{tv0}[]{\small $0$}
\psfrag{tv14}[]{\small $\frac14$}
\psfrag{tv12}[]{\small $\frac12$}
\begin{center}
\includegraphics[height=5.8cm]{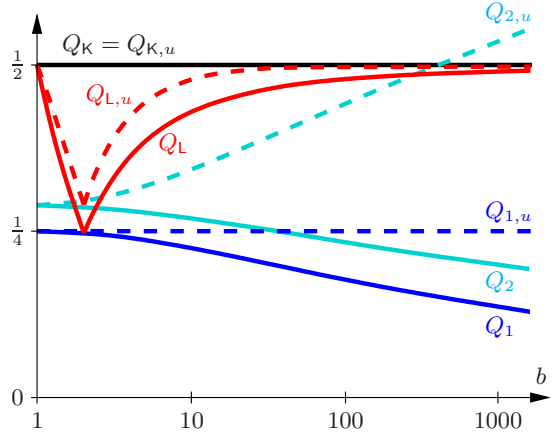}
\end{center}
\caption{Comparing the quantization coefficients $Q_{*}$
  (solid curves) and uniform quantization coefficients $Q_{*,u}$ (broken curves)
  of $\beta_b$, for $\ast = {\sf L}$ (red), $\ast =1,2$ (blue),
  and $\ast = {\sf K}$ (black), respectively; see also Figure \ref{fig3}.}\label{fig4}
\end{figure}

\begin{rem}
In the context of Benford's Law, $\I = [1,b]$, and since $S_b < b$
always, it may seem more natural to study the approximation problem not on all of
$\cP$, but rather on the (dense) subset $\widetilde{\cP}:= \bigl\{\mu \in
\cP: \mu (\{b\}) = 0\bigr\}$. Clearly, $d_{\sf L}$ and $d_r$ both
metrize the weak topology on $\widetilde{\cP}$ but are not
complete. (By contrast, $d_{\sf K}$ is complete but not separable, and
induces a finer topology.) Since $\widetilde{\cP}$ is a
$G_{\delta}$-set in $\cP$, a classical theorem \cite[Thm.2.5.4]{D}
yields, for instance,
$$
\widetilde{d}(\mu, \nu) = \int_0^1 \lt|G_{\mu} - G_{\nu}\rt| +
\sum\nolimits_{k=1}^{\infty}  \frac{2^{-k} \left| \int_{1-k^{-1}}^1
    \lt(G_{\mu} - G_{\nu} \rt) \right|}{
\int_{1-k^{-1}}^1  G_{\mu}   \int_{1-k^{-1}}^1 G_{\nu}
+ \left| \int_{1-k^{-1}}^1
    \lt(G_{\mu} - G_{\nu}  \rt)\right|} ,
$$
with $G_{\mu} = b - F_{\mu}^{-1}$, $G_{\nu} = b - F_{\nu}^{-1}$,
as an equivalent complete, separable metric on
$\widetilde{\cP}$. However, $\widetilde{d}$ appears to be quite unwieldy,
and the authors do not know of an equivalent {\em complete\/} metric
on $\widetilde{\cP}$ for which explicit results similar to those in
Sections \ref{secL} and \ref{secr} could be established.

Also, it is readily confirmed that, given any $\mu \in
\widetilde{\cP}$, there exists a best (or best uniform)
$d_{*}$-approximation $\delta_{\bullet}^{\bullet, n}\in
\widetilde{\cP}$ (or $\delta_{\bullet}^{u_n}\in
\widetilde{\cP}$), i.e., these approximation problems always have a
solution in $\bigl(\widetilde{\cP}, d_*\bigr)$, notwithstanding the fact that
the latter space is not complete (if $* = {\sf L}, r$) or not
separable (if $* = {\sf K}$).
\end{rem}

For Benford's Law, as seen above, all best (or best uniform)
approximations considered converge at the same rate, namely
$(n^{-1})$; the same is true for the $\mbox{\tt Beta}(2,1)$
distribution whenever $1\le r < 2$. These are not coincidences. Rather, for many other
probability metrics $n^{-1}$ turns out to yield the correct order of magnitude
of the $n$-th quantization error as well. Specifically,
consider a metric $d$ on $\cP$ for which
\begin{eqnarray}\label{eq61}
a_1 \| F_{\mu}^{s_1} - F_{\nu}^{s_1} \|_1 & \le &  d(\mu, \nu) \\
& \le & a_2
\left( 
\epsilon  \|F_{\mu}^{s_2} - F_{\nu}^{s_2}\|_{\infty} +  (1-\epsilon) \| F_{\mu}^{-1} -
F_{\nu}^{-1}\|_{\infty}  \right) \quad \forall \mu, \nu \in \cP \, , \nonumber
\end{eqnarray}
with positive constants $a_1, a_2 , s_1, s_2$, and $\epsilon \in \{0, 1 \}$; see, e.g., \cite{BA,R,RKSF} for examples and properties of such metrics. Note that
validity of (\ref{eq61}) causes $d$ to metrize a topology at least as fine as the weak topology,
and clearly (\ref{eq61}) holds for any $d = d_*$. The latter fact, together
with \cite[Thm.6.2]{GL} yields a simple observation regarding the
prevalence of the rate $(n^{-1})$.

\begin{prp}\label{prop66}
Let $d$ be a metric on $\cP$ satisfying {\rm (\ref{eq61})}. Then, for
every $\mu \in \cP$,
$$
\limsup\nolimits_{n\to \infty} n \inf\nolimits_{x\in \Xi_n , p \in
  \Pi_n} d\bigl( \mu, \delta_x^p \bigr) < + \infty \, ,
$$
and if $\mu$ is non-singular (w.r.t.\ $\lambda$) then also
$$
\liminf\nolimits_{n\to \infty} n \inf\nolimits_{x\in \Xi_n , p \in
  \Pi_n} d \bigl( \mu, \delta_x^p\bigr) >0\, .
$$
\end{prp}

\rb
{\rm (i)} Apart from $d_*$, examples of familiar probability metrics
that satisfy (\ref{eq61}) include the discrepancy distance $\sup_{I
  \subset \R} |\mu (I) - \nu (I)|$ and the $L^r$-distance $\|F_{\mu} -
F_{\nu}\|_r$ between distribution functions \cite{R}. For the important
Prokhorov distance, validity of the right-hand inequality in
(\ref{eq61}) appears to be unknown \cite{GS}, but best
approximations are suspected to converge at the rate $\bigl( n^{-1} \bigr)$
regardless \cite[Sec.4]{GL1}. Also, $(n^{-1})$ is
established in \cite{DV} as the universal rate of convergence for best approximations under Orlicz
norms, which contains $d_r$ as a special case.

{\rm (ii)} In \cite[Sec.4.2]{RKSF}, for any $a\ge 0$, the $a$-L\'evy
distance
$$
d_{{\sf L}_{a}}(\mu,\nu)= \inf \lt\{ y\ge0: F_{\mu} (\cdot-a y) -
y \le F_{\nu} \le F_{\mu} (\cdot+ a y) + y \rt\}
$$
is considered. Every $d_{{\sf L}_a}$ satisfies (\ref{eq61}), and
$d_{{\sf L}_0} = d_{\sf K}$, $d_{{\sf L}_1} = \omega^{-1}
d_{\sf L}$. Usage of $a$-L\'evy distances may
enable a unified treatment of the results in Sections \ref{secL} and \ref{secK}.

{\rm(iii)} Under additional assumptions on $\mu$, the value of $ n\inf_{x\in \Xi_n} d(\mu,
\delta_{x}^{u_n})$ can similarly be bounded above and
below by positive constants \cite[Thm.5.15]{XB}.
\re

\medskip

Finally, it is worth pointing out that, though motivated here by Benford's Law,
compactness of the interval $\I$ was assumed largely for convenience,
and can easily be dispensed with for many of the general results in
this article. For instance, if $\I$ is (closed but) unbounded then
(\ref{eq1}), with $\omega =1$, still yields $d_{\sf L}$ as a complete,
separable metric inducing the weak topology on $\cP$, though the
latter no longer is compact. Clearly, Theorem \ref{thm3e} is valid in
this situation, as (\ref{eq3100}) holds for $f=F_{\mu}$ and any
interval $I \subset \overline{\R}$. Even though (\ref{eq3100}) may
fail for $f=F_{\mu}^{-1}$ when $\mbox{\rm
  supp}\, \mu$ is unbounded, it is readily checked that nevertheless
the conclusions of Proposition \ref{prop3c} remain intact for
$\ell_{F_{\mu}^{-1}, I}$, provided that $I \subset [0,1]$ but $I \ne
\{0\}$ and $I \ne \{ 1 \}$. With $\ell_{F^{-1}_{\mu}, \{0\}}^*
:=\ell_{F^{-1}_{\mu}, \{1\}}^* := 0$, then, Proposition \ref{prop3f}
holds verbatim, and so does Theorem \ref{thm3h}. Analogously, Theorems
\ref{thloc-K}, \ref{thwei-K}, and \ref{bestK} all can be seen to be
correct, with the definition of ${\sf K}_{\bullet}(p)$ understood to assume that
$p_{,1} p_{,n}>0$. By contrast, the classical $L^1$-Kantorovich
distance $d_1 (\mu, \nu) = \|F_{\mu}^{-1} - F_{\nu}^{-1}\|_1$ is
defined only on the (dense) subset $\cP_1 = \left\{\mu \in \cP : \int_{\I} |x| \, {\rm d}\mu
(x ) < + \infty \right\}$ where it metrizes a topology finer than
the weak topology. Still, with $\cP$ replaced by $\cP_1$, Proposition \ref{prop4a} also remains intact; see, e.g.,
\cite[Sec.5]{XB}. Note that the sequence $\bigl( nd_* (\mu,
\delta_{\bullet}^{u_n})\bigr)$ is bounded when $* = {\sf L}, {\sf K}$
because $d_{\sf L} \le d_{\sf K}$, whereas $\bigl( nd_1 (\mu,
\delta_{\bullet}^{\bullet, n})\bigr)$ may decay arbitrarily slowly; see
\cite[Thm.5.32]{XB}. For a simple application of these results to a
probability measure with unbounded support, let $\mu$ be the standard exponential
distribution, i.e., $F_{\mu}(x) = \max \{0, 1 -
e^{-x}\}$. Calculations quite similar to the ones shown earlier for
Benford's Law yield
$$
\lim\nolimits_{n\to\infty}nd_{\sf
  L}\lt(\mu,\delta_{\bullet}^{\bullet,n}\rt)=\frac{\log2}{2} \, , 
\quad \lim\nolimits _{n\to\infty}nd_{\sf
  L}\lt(\mu,\delta_{\bullet}^{u_n}\rt)=\frac{1}{2}\, ,
$$
whereas
$$
\lim\nolimits_{n\to\infty} n d_1\lt(\mu,\delta_{\bullet}^{\bullet, n}\rt)=1  \quad \mbox{but}
\quad
\lim\nolimits_{n\to\infty}\frac{n}{\log n}d_1
(\mu,\delta_{\bullet}^{u_n})=\frac14 \, ,
$$
and clearly $nd_{\sf K} (\mu , \delta_{\bullet}^{\bullet, n}) =
nd_{\sf K} (\mu , \delta_{\bullet}^{u_n})=\frac12$ for all $n$. Even
though $\mu$ has finite moments of all orders, there exist probability
metrics $d$ for which $\bigl( nd (\mu , \delta_{\bullet}^{\bullet,
  n})\bigr)$ is unbounded; see \cite[Ex.5.1(d)]{GL1}.

\subsection*{Acknowledgements}
The first author was partially supported by an {\sc Nserc} Discovery
Grant. Both authors gratefully acknowledge helpful suggestions made by
F.\ Dai, B. Han, T.P.\ Hill, and an anonymous referee.

\end{document}